\documentclass[11pt]{amsart}
\usepackage{amsthm,amssymb,mathrsfs,mathtools,empheq}

\usepackage[T1]{fontenc}

\usepackage[final,notref,notcite]{showkeys}
\mathtoolsset{showonlyrefs}
\usepackage{fullpage,color}

\newtheorem{mainthm}{Theorem}
\newtheorem{theorem}{Theorem}[section]
\newtheorem*{theorem*}{Theorem}
\newtheorem{corollary}[theorem]{Corollary}

\newtheorem{lemma}[theorem]{Lemma}
\newtheorem{proposition}[theorem]{Proposition}
\newtheorem*{proposition*}{Proposition}

\newtheorem*{conjecture*}{Conjecture}

\theoremstyle{definition}

\newtheorem{remark}[theorem]{Remark}

\numberwithin{equation}{section}

\def\bR {\mathbb{R}}

\def\cA {\mathcal{A}}

\def\cE {\mathcal{E}}
\def\cF {\mathcal{F}}

\def\cY {\mathcal{Y}}
\def\cZ {\mathcal{Z}}

\def\scrL{\mathscr{L}}

\def\grad {{\nabla}}

\def\rstr {{\big |}}

\def\la {\langle}
\def\ra {\rangle}

\newcommand{\tx}[1]{\mathrm{#1}}

\newcommand{\wt}[1]{\widetilde{#1}}
\newcommand{\bs}[1]{\boldsymbol{#1}}

\newcommand{\supp}{\operatorname{supp}}

\newcommand{\Id}{\operatorname{Id}}

\newcommand{\eee}{\mathrm e}

\newcommand{\ud}{\mathrm{\,d}}
\newcommand{\vd}{\mathrm{d}}
\newcommand{\vD}{\mathrm{D}}
\newcommand{\dd}[1]{{\frac{\vd}{\vd{#1}}}}

\newcommand{\enorm}{{\dot H^1 \times L^2}}

\newcommand{\uln}[1]{{\underline{ #1 }}}

\newcommand{\yp}{{\cY}^+}
\newcommand{\ym}{{\cY}^-}
\newcommand{\ypl}{{\cY}^+_{\lambda}}
\newcommand{\yml}{{\cY}^-_{\lambda}}

\title{Bounds on the speed of type II blow-up for the energy critical wave equation in the radial case}
\author{Jacek Jendrej}
\address{\'Ecole Polytechnique, CMLS, 91128 Palaiseau, France}
\email{jacek.jendrej@polytechnique.edu}

\begin{document}

\maketitle
\begin{abstract}
We consider the focusing energy-critical wave equation in space dimension $N \in\{3, 4, 5\}$ for radial data.
We study type II blow-up solutions which concentrate one bubble of energy.
It is known that such solutions decompose in the energy space as a sum of the bubble and an asymptotic profile.
We prove bounds on the blow-up speed in the case when the asymptotic profile is sufficiently regular.
These bounds are optimal in dimension $N = 5$.
We also prove that if the asymptotic profile is sufficiently regular, then it cannot be strictly negative at the origin.
\end{abstract}
\section{Introduction}
\label{sec:intro}
\subsection{Setting of the problem}
Let $N \in \{3, 4, 5\}$ be the dimension of the space. For $\bs u_0 = (u_0, \dot u_0) \in \cE := \dot H^1(\bR^N) \times L^2(\bR^N)$, define the \emph{energy functional}
\begin{equation*}
  E(\bs u_0) = \int\frac 12|\dot u_0|^2 + \frac 12|\grad u_0|^2 - F(u_0)\ud x,
\end{equation*}
where $F(u_0) := \frac{N-2}{2N}|u_0|^\frac{2N}{N-2}$. Note that $E(\bs u_0)$ is well-defined due to the Sobolev Embedding Theorem.
The differential of $E$ is $\vD E(\bs u_0) = (-\Delta u_0 - f(u_0), \dot u_0)$, where $f(u_0) = |u_0|^\frac{4}{N-2}u_0$.

We consider the Cauchy problem for the energy critical wave equation:
\begin{empheq}{equation}
  \label{eq:nlw}
\bigg\{
  \begin{aligned}
  \partial_t \bs u(t) &= J\circ \vD E(\bs u(t)), \\
  \bs u(t_0) &= \bs u_0 \in \cE.
\end{aligned}
\tag{NLW}
\end{empheq}
Here, $J := \begin{pmatrix}0 & \Id \\ -\Id & 0 \end{pmatrix}$ is the natural symplectic structure. This equation is often written in the form
  \begin{equation*}
    \partial_{tt} u = \Delta u + f(u).
  \end{equation*}

  Equation \eqref{eq:nlw} is locally well-posed in the space $\cE$, see for example \cite{GSV92} and \cite{ShSt94} (the defocusing case),
  as well as a complete review of the Cauchy theory in \cite{KeMe08}.
In particular, for any initial data $\bs u_0 \in \cE$ there exists a maximal time of existence $(T_-, T_+)$, $-\infty \leq T_- < t_0 < T_+ \leq +\infty$,
and a unique solution $\bs u \in C((T_-, T_+); \cE)$. In addition, the energy $E$ is a conservation law.
In this paper we always assume that the initial data is radially symmetric. This symmetry is preserved by the flow.

For functions $v \in \dot H^1$, $\dot v \in L^2$, $\bs v = (v, \dot v)\in \cE$ and $\lambda > 0$, we denote
\begin{equation*}
  v_\lambda(x) := \frac{1}{\lambda^{(N-2)/2}} v\big(\frac{x}{\lambda}\big), \qquad \dot v_\uln\lambda(x) := \frac{1}{\lambda^{N/2}} \dot v\big(\frac{x}{\lambda}\big),\qquad\bs v_\lambda(x) := \big(v_\lambda, \dot v_\uln\lambda\big).
\end{equation*}

A change of variables shows that
\begin{equation*}
  E\big((\bs u_0)_\lambda\big) = E(\bs u_0).
\end{equation*}
Equation~\eqref{eq:nlw} is invariant under the same scaling. If $\bs u = (u, \dot u)$ is a solution of \eqref{eq:nlw} and $\lambda > 0$, then
$
t \mapsto \bs u\big((t-t_0)/\lambda\big)_\lambda
$ is also a solution
with initial data $(\bs u_0)_\lambda$ at time $t = 0$.
This is why equation~\eqref{eq:nlw} is called \emph{energy-critical}.

A fundamental object in the study of \eqref{eq:nlw} is the family of stationary solutions $(u, \partial_t u) = \pm\bs W_\lambda = (\pm W_\lambda, 0)$, where
\begin{equation*}
  W(x) = \Big(1 + \frac{|x|^2}{N(N-2)}\Big)^{-(N-2)/2}.
\end{equation*}
The functions $W_\lambda$ are called \emph{ground states}. 

In general the energy $E$ does not control the norm $\|\cdot\|_\cE$, and indeed this norm can tend to $+\infty$ in finite time, which is referred to as \emph{type I blow-up}.
In odd space dimensions and for superconformal nonlinearities (which includes the energy-critical case)
Donninger and Sch\"orkhuber \cite{DoSc14}, \cite{DoSc15p} described large sets of initial data leading to this kind of blow-up.

It can also happen that in finite time the solution leaves every compact set of $\cE$, the norm $\|\cdot\|_\cE$ staying bounded, which is referred to as \emph{type II blow-up}.
In dimension $N = 3$ in the radial case one of the consequences of the classification result of Duyckaerts, Kenig and Merle \cite{DKM4} is that any blow-up solution
is either of type I or of type II. This is unknown in other cases.

A particular type of type II blow-up occurs when the solution $\bs u(t)$ stays close to the family of ground states $\bs W_\lambda$ and $\lambda \to 0$.
In this situation we call $\bs W_\lambda$ the \emph{bubble of energy} and we say that $\bs u(t)$ blows up by concentration of one bubble of energy.
We have the following fundamental result proved first by Duyckaerts, Kenig and Merle \cite{DKM1} for $N = 3$,
by the same authors \cite{DKM2} for $N = 5$ and by C\^ote, Kenig, Lawrie and Schlag \cite{CKLS14p} for $N = 4$:
\begin{theorem*}[\cite{DKM1}, \cite{DKM2}, \cite{CKLS14p}]
  Let $\bs u(t)$ be a radial solution of \eqref{eq:nlw} which blows up at $t = T_+$
  by concentration of one bubble of energy. Then there exist $\bs u^*_0 \in \cE$ and $\lambda \in C([t_0, T_+), (0, +\infty))$
    such that
    \begin{equation}
      \label{eq:DKM1}
      \lim_{t\to T_+}\|\bs u(t) - \bs W_{\lambda(t)} - \bs u^*_0\|_\cE = 0,\qquad \lim_{t \to T_+} (T_+ - t)^{-1}\lambda(t) = 0.
    \end{equation} \qed
\end{theorem*}
In this context the function $\bs u^*_0$ is called the \emph{asymptotic profile}.
Note that in \cite{DKM2} a more general, non-radial version of the above theorem was proved for $N \in \{3, 5\}$.

Solutions verifying \eqref{eq:DKM1} were first constructed in dimension $N = 3$ by Krieger, Schlag and Tataru \cite{KrScTa09}, who obtained all possible
polynomial blow-up rates $\lambda(t) \sim (T_+ -t)^{1+\nu}$, $\nu > 0$.
For $N = 4$ smooth solutions blowing up at a particular rate were constructed by Hillairet and Rapha\"el \cite{HiRa12}.
For $N = 5$ the author proved in \cite{moi15p} that for any radially symmetric asymptotic profile $\bs u^*_0 \in H^4 \times H^3$
such that $u^*_0(0) > 0$, there exists a solution $\bs u(t)$ such that \eqref{eq:DKM1} holds.
For these solutions the concentration speed of the bubble is
\begin{equation}
  \label{eq:con-speed}
  \lambda(t) \sim u^*_0(0)^2 (T_+ - t)^4.
\end{equation}
In the same article, solutions with blow-up rate $(T_+ - t)^{1+\nu}$ for $\nu > 8$ were constructed,
with $\nu$ explicitely related to the asymptotic behaviour of $\bs u^*_0$ at $x = 0$.
\subsection{Statement of the results}
In the present paper we continue the investigation of the relationship between the behaviour of $\bs u^*_0$ at $x = 0$
and possible blow-up speeds, still in the special case when the asymptotic profile $\bs u^*_0$ is sufficiently regular.
We prove the following result.
\begin{mainthm}
  \label{thm:loi}
 Let $N \in \{3, 4, 5\}$ and $s > \frac{N-2}{2}$, $s \geq 1$. Let $\bs u_0^* = (u_0^*, \dot u_0^*) \in H^{s+1} \times H^s$ be a radial function.
  Suppose that $\bs u$ is a radial solution of \eqref{eq:nlw} such that
  \begin{equation}
    \label{eq:loi}
    \lim_{t \to T_+}\|\bs u(t) - \bs W_{\lambda(t)} - \bs u_0^*\|_\cE = 0, \qquad \lim_{t \to T_+} \lambda(t) = 0,\qquad T_+ < +\infty.
  \end{equation}
 There exists a constant $C > 0$ depending on $\bs u^*_0$ such that:
  \begin{itemize}
    \item if $N \in \{4, 5\}$, then for $T_+ - t$ sufficiently small there holds
      \begin{equation}
        \label{eq:loi-borne}
        \lambda(t) \leq C (T_+ - t)^\frac{4}{6-N}.
      \end{equation}
  \item if $N = 3$, then there exists a sequence $t_n \to T_+$ such that
      \begin{equation}
        \label{eq:loi-borne-N3}
        \lambda(t_n) \leq C (T_+ - t_n)^\frac{4}{6-N}.
      \end{equation}
  \end{itemize}
\end{mainthm}
\begin{remark}
  \label{rem:Linf}
  Let $\bs u^* = (u^*, \dot u^*)$ be the solution of \eqref{eq:nlw} such that $\bs u^*(T_+) = \bs u^*_0$ and suppose that $0 \in \supp\bs u^*_0$.
  We will prove that there exists a universal constant $C_0$ such that in the above theorem one can take
  \begin{equation*}
    C = C_0\|u^*\|_{L^\infty((T_+ - \rho, T_+)\times B(0, \rho))}^\frac{2}{6-N},
  \end{equation*}
  where $\rho > 0$ is arbitrary and $B(0, \rho)$ is the ball of centre $0$ and radius $\rho$ in $\bR^N$.
  Notice that $u^* \in L^\infty((T_+ - \rho, T_+)\times \bR^N)$ by Appendix~\ref{sec:cauchy} and the Sobolev Embedding Theorem.

  If $0 \notin \supp \bs u^*_0$, then blow-up cannot occur, as follows from the classification
  of solutions of \eqref{eq:nlw} at energy level $E(\bs W)$ by Duyckaerts and Merle \cite{DM08}.
\end{remark}
\begin{remark}
  \label{rem:N3}
  In the case $N = 3$ we will prove that for $T_+ - t$ small enough there holds
  \begin{equation}
    \label{eq:moyenne}
    \int_t^{T_+}\frac{\ud \tau}{\sqrt{\lambda(\tau)}} \geq \frac{3}{\sqrt C}(T_+ - t)^\frac 13,
  \end{equation}
  which immediately implies \eqref{eq:loi-borne-N3}.

  If we assume that $\bs u^* \in H^3 \times H^2$, then \eqref{eq:loi-borne} holds also in the case $N = 3$, see Remark~\ref{rem:N3-cont}.
  I~believe that the proof of \eqref{eq:loi-borne-N3} given here could be adapted to cover the case $1 > s > \frac 12$.
\end{remark}
\begin{remark}
  \label{rem:optimal}
  In dimension $N = 5$ the bound \eqref{eq:loi-borne} is optimal, see \eqref{eq:con-speed}.
It is not clear if the bounds are optimal for $N \in \{3, 4\}$, due to slow decay of the bubble.
\end{remark}
\begin{remark}
  \label{rem:rough}
A natural problem is to determine sharp bounds for the blow-up speed in the case of less regular $\bs u^*_0$.
The method used in this paper allows to obtain some bounds for example in the case $1 \leq s< \frac 32$ in dimension $N = 5$,
but they are not optimal and I will not pursue this direction here.
\end{remark}
  In the case $u^*_0(0) = 0$ one could obtain various bounds depending on the asymptotics of $\bs u^*_0$ at $x = 0$,
  but this will not be considered in the present paper.
  Along the same line, one can ask if the sign of $u^*_0(0)$ is relevant in the case when $u^*_0(0) \neq 0$.
  It turns out that it is, but unfortunately our method requires the additional assumption $\bs u^*_0 \in H^3 \times H^2$:
  \begin{mainthm}
    \label{thm:neg}
    Let $N \in \{3, 4, 5\}$. Let $\bs u_0^* = (u_0^*, \dot u_0^*) \in H^3 \times H^2$ be a radial function such that
    \begin{equation}
      \label{eq:ustar-neg}
        u^*_0(0) < 0.
    \end{equation}
  There exist no radial solutions of \eqref{eq:nlw} such that
  \begin{equation}
    \label{eq:neg}
    \lim_{t \to T_+}\|\bs u(t) - \bs W_{\lambda(t)} - \bs u_0^*\|_\cE = 0, \qquad \lim_{t \to T_+} \lambda(t) = 0,\qquad T_+ < +\infty.
  \end{equation}
  \end{mainthm}
\begin{remark}
  \label{rem:radial}
  I expect that Theorems~\ref{thm:loi} and \ref{thm:neg} could be proved by similar methods without the assumption of $\bs u^*_0$ being radial.
\end{remark}
\subsection{Related results}
  \label{ssec:related}
  The problem of existence of an asymptotic profile at blow-up might be seen as a version of the classical question of asymptotic stability of solitons
  in the case when finite-time blow-up occurs. Decompositions of type \eqref{eq:DKM1} in suitable topologies are believed to hold for many models,
  but establishing this rigourously is a challenging problem.
  Historically, the study of finite type blow-up in the Hamiltonian setting received the most attention probably in the case of nonlinear Schr\"odinger equations (NLS).
  For the mass-critical NLS the conformal invariance leads to explicit blow-up solutions $S(t)$ with the asymptotic profile $u^* \equiv 0$.
  Bourgain and Wang \cite{BoWa97} constructed examples of blow-up solutions with $u^*$ regular and non-zero, the speed of blow-up however being the same as for $S(t)$.
  This is not a coincidence, as shown by a classification result of Merle and Rapha\"el \cite{MeRa05}.

  For the critical gKdV equation Martel, Merle and Rapha\"el \cite{MMR12-1p} proved that if the initial data decays sufficiently fast, then there is only one possible blow-up speed,
  given by the minimal mass blow-up solution. However, without the decay assumption other blow-up speeds are possible, as shown by the same authors in \cite{MMR12-3p}.

  These are the main two examples of the heuristic principle that the size of the interaction of the bubble with the rest of the solution influences or even determines the speed of blow-up.
  In the present paper we try to investigate this phenomenon in the energy-critical setting.

  Finally, let us mention that the problem of understanding the possible blow-up speeds is not limited to type II blow-up for critical equations. For example, for the subconformal and conformal
  NLW this was considered in the works of Merle and Zaag \cite{MeZa03}, \cite{MeZa05}.

  \subsection{Outline of the proof}

  Our proofs of Theorems~\ref{thm:loi} and \ref{thm:neg} are based on the following computation that we present here formally.

  Let $\bs u: [t_0; T_+) \to \cE$ be a solution of \eqref{eq:nlw} which satisfies \eqref{eq:DKM1}.
    At blow-up time, the energy of the bubble is completely decoupled from the energy of the asymptotic profile, hence
    \begin{equation}
      \label{eq:energy}
      E(\bs u) = E(\bs u^*_0) + E(\bs W_\lambda) = E(\bs u^*_0) + E(\bs W).
    \end{equation}
    Let $\bs u^*$ be the solution of \eqref{eq:nlw} with the initial data $\bs u^*(T_+) = \bs u^*_0$.
    Decompose $\bs u(t) = \bs W_{\lambda(t)} + \bs u^*(t) + \bs g(t)$. The \emph{modulation parameter} $\lambda$ is determined by a suitable orthogonality condition,
    and a standard procedure shows that $|\lambda'(t)| \lesssim \|\bs g(t)\|_\cE$.

    From the Taylor formula we obtain
    \begin{equation*}
      E(\bs u) = E(\bs u^* + \bs W_\lambda) + \la \vD E(\bs u^* + \bs W_\lambda), \bs g\ra + \frac 12 \la \vD^2 E(\bs u^* + \bs W_\lambda)\bs g, \bs g\ra + O(\|\bs g\|_\cE^3).
    \end{equation*}
    \paragraph{Step 1.} An explicit key computation shows that $$E(\bs u^* + \bs W_\lambda) - E(\bs u^*) - E(\bs W) \gtrsim -u^*_0(0)\lambda^\frac{N-2}{2}.$$
    It is clear that the sign of $u^*_0(0)$ is decisive.
    \paragraph{Step 2.} Near blow-up time $\bs u^*$ weakly interacts with $\bs W_\lambda$ and $\vD E(\bs W_\lambda) = 0$. This allows to replace $\la \vD E(\bs u^* + \bs W_\lambda), \bs g\ra$
    by $\la \vD E(\bs u^*), \bs g\ra$. Using the Hamiltonian structure it is seen that this quantity is, at first order in $\bs g$, a conservation law.
    Estimating some error terms we conclude that this term can be neglected.
    \paragraph{Step 3.} Let us suppose for a moment that $\vD^2 E(\bs W)$ is a coercive functional in the sense that $\la \vD^2 E(\bs u^* + \bs W_\lambda)\bs g, \bs g\ra \gtrsim \|\bs g\|_\cE^2$.
    Using \eqref{eq:energy} and the two preceding steps we find $|\lambda'|^2 \lesssim \|\bs g\|_\cE^2 \lesssim u^*_0(0)\lambda^\frac{N-2}{2}$. In the case $u^*_0(0) < 0$ this is contradictory,
    and in the case $u^*_0(0) > 0$ the conclusion follows by integrating the differential inequality for $\lambda$.

~

Strictly speaking, $\vD^2 E(\bs W)$ is not a coercive functional, and much of the proof is devoted to controlling
the negative directions, which are related to the eigendirections of the flow linearized around $\bs W$.
Clarifying the second step above is another major technical difficulty of this paper.
    \subsection{Acknowledgements}
  This paper has been prepared as a part of my \mbox{Ph.~\!D.} under supervision of Y.~Martel and F.~Merle.
  I would like to thank my supervisors for their constant support and many helpful discussions.
  The author has been supported by the ERC~grant $291214$ BLOWDISOL.

\subsection{Notation}
\label{ssec:notation}
We introduce the inifinitesimal generators of scale change
$$\Lambda_s := \big(\frac{N}{2} - s\big) + x\cdot\grad.$$
For $s = 1$ we omit the subscript and write $\Lambda = \Lambda_1$.
We denote $\Lambda_\cE$, $\Lambda_\cF$ and $\Lambda_{\cE^*}$ the inifinitesimal generators of the scaling which is critical for a given norm,
that is
$$\Lambda_\cE = (\Lambda, \Lambda_0),\quad \Lambda_\cF = (\Lambda_0, \Lambda_{-1}),\quad \Lambda_{\cE^*} = (\Lambda_{-1}, \Lambda_0).$$

The dimension of the space will be denoted $N$. The domain of the function spaces is always $\bR^N$. We introduce the following notation for some frequently used function spaces:
$X^s := \dot H^{s+1} \cap \dot H^1$ for $s \geq 0$, $\cE := \enorm$, $\cF := L^2 \times \dot H^{-1}$.
The bracket $\la\cdot, \cdot\ra$ denotes the distributional pairing and the scalar product in the spaces $L^2$, $L^2 \times L^2$.
Notice that $\cE^* \simeq \dot H^{-1}\times L^2$ through the natural isomorphism induced by $\la \cdot, \cdot\ra$.

For a function space $\cA$, $O_\cA(m)$ denotes any $a \in \cA$ such that $\|a\|_\cA \leq Cm$ for some constant $C > 0$.
For positive quantities $m_1$ and $m_2$ we write $m_1 \lesssim m_2$ for $m_1 = O(m_2)$ and $m_1 \sim m_2$ for $m_1 \lesssim m_2 \lesssim m_1$.
We denote $B_\cA(x_0, \delta)$ the open ball of center $x_0$ and radius $\delta$ in the space $\cA$.
If $\cA$ is not specified, it means that $\cA = \bR$.

\section{The proofs}
\subsection{Properties of the linearized operator}
\label{ssec:lin}
%
%
Linearizing $-\Delta u - f(u)$ around $W$, $u = W + g$, we obtain a Schr\"odinger operator $$Lg = (-\Delta - f'(W))g.$$
Notice that $L(\Lambda W) = \dd\lambda\rstr_{\lambda = 1}\big(-\Delta W_\lambda - f(W_\lambda)\big) = 0$.
It is known that $L$ has exactly one strictly negative simple eigenvalue which we denote $-\nu^2$ (we take $\nu > 0$).
We denote the corresponding positive eigenfunction $\cY$, normalized so that $\|\cY\|_{L^2} = 1$.
By elliptic regularity $\cY$ is smooth and by Agmon estimates it decays exponentially.
Self-adjointness of $L$ implies that
\begin{equation}
  \label{eq:YLW}
  \la \cY, \Lambda W\ra = 0.
\end{equation}

We define
\begin{equation}
  \label{eq:Ya}
  \ym := \big(\frac 1\nu\cY, -\cY\big),\qquad \yp := \big(\frac 1\nu\cY, \cY\big),\qquad \alpha^- := \frac 12(\nu\cY, -\cY),\qquad \alpha^+ := \frac 12(\nu\cY, \cY).
\end{equation}

We have $J\circ\vD^2 E(\bs W) = \begin{pmatrix} 0 & \Id \\ -L & 0\end{pmatrix}$. A short computation shows that
\begin{equation}
  \label{eq:eigenvect}
  J\circ\vD^2 E(\bs W)\ym = -\nu \ym,\qquad J\circ\vD^2 E(\bs W)\yp = \nu \yp
\end{equation}
and
\begin{equation}
  \label{eq:eigencovect}
  \la\alpha^-, J\circ\vD^2 E(\bs W)\bs g\ra = -\nu\la\alpha^-, \bs g\ra,\qquad \la\alpha^+, J\circ\vD^2 E(\bs W)\bs g\ra = \nu\la\alpha^+, \bs g\ra,\qquad \forall \bs g\in\cE.
\end{equation}
Notice that $\la \alpha^-, \ym\ra = \la \alpha^+, \yp\ra = 1$ and $\la \alpha^-, \yp\ra = \la \alpha^+, \ym\ra = 0$.

The rescaled versions of these objects are
\begin{equation}
  \label{eq:Yal}
  \yml := \big(\frac 1\nu\cY_\lambda, -\cY_\uln\lambda\big),\qquad \ypl := \big(\frac 1\nu\cY_\lambda, \cY_\uln\lambda\big),
  \qquad \alpha^-_\lambda := \frac 12\big(\frac{\nu}{\lambda}\cY_\uln\lambda, -\cY_\uln\lambda\big),\qquad \alpha^+_\lambda := \frac 12\big(\frac{\nu}{\lambda}\cY_\uln\lambda, \cY_\uln\lambda\big).
\end{equation}
The scaling is chosen so that $\la\alpha_\lambda^-, \ym_\lambda\ra = \la\alpha_\lambda^+, \yp_\lambda\ra = 1$.
We have
\begin{equation}
  \label{eq:eigenvectl}
  J\circ\vD^2 E(\bs W_\lambda)\ym_\lambda = -\frac{\nu}{\lambda} \ym_\lambda,\qquad J\circ\vD^2 E(\bs W_\lambda)\yp_\lambda = \frac{\nu}{\lambda} \yp_\lambda
\end{equation}
and
\begin{equation}
  \label{eq:eigencovectl}
  \la\alpha_\lambda^-, J\circ\vD^2 E(\bs W_\lambda)\bs g\ra = -\frac{\nu}{\lambda}\la\alpha_\lambda^-, \bs g\ra,
  \qquad \la\alpha_\lambda^+, J\circ\vD^2 E(\bs W_\lambda)\bs g\ra = \frac{\nu}{\lambda}\la\alpha_\lambda^+, \bs g\ra,\qquad \forall \bs g\in\cE.
\end{equation}

Let $\cZ$ be a $C_0^\infty$ function such that
\begin{equation}
  \label{eq:Z}
    \la \cZ, \Lambda W\ra > 0, \qquad \la \cZ, \cY\ra = 0
\end{equation}
(the first condition is the essential one and the second allows to simplify some computations).
We recall the following result.
\begin{proposition}[{\cite[Lemma~6.1]{moi15p}}, {\cite[Proposition~5.5]{DM08}}]
  \label{prop:coerL}
  There exists a constant $c_L > 0$ such that
  \begin{equation*}
    v\in\dot H^1\ \text{radial},\quad \la \cY, v\ra = \la \cZ, v\ra = 0 \quad\Rightarrow\quad \frac 12 \la v, Lv\ra \geq c_L\|v\|_{\dot H^1}^2.
  \end{equation*} \qed
\end{proposition}

\begin{lemma}
  \label{lem:coer}
  There exists a constant $c > 0$ such that if $\|\bs V - \bs W_\lambda\|_\cE < c$, then for all $\bs g \in \cE$ such that $\la \cZ_\uln\lambda, g\ra = 0$ there holds
  \begin{equation*}
    \frac 12 \la \vD^2 E(\bs V)\bs g, \bs g\ra +2\big(\la \alpha^-_\lambda, \bs g\ra^2 + \la \alpha^+_\lambda, \bs g\ra^2\big) \gtrsim \|\bs g\|_\cE^2.
  \end{equation*}
\end{lemma}
\begin{proof}
  We have
  $$
  \la \vD^2 E(\bs V)\bs g, \bs g\ra = \la \vD^2 E(\bs W_\lambda)\bs g, \bs g\ra + \int \big(f'(V) - f'(W_\lambda)\big)|g|^2\ud x.
  $$
  By H\"older, the last integral is $\lesssim c\|\bs g\|_\cE^2$, hence it suffices to prove the lemma with $\bs V = \bs W_\lambda$.
  Without loss of generality we can assume that $\lambda = 1$.
  We will show the following stronger inequality:
  \begin{equation}
    \label{eq:coer-lin-exact}
    \frac 12\la\vD^2 E(\bs W)\bs g, \bs g\ra + 2\la\alpha^-, \bs g\ra\cdot\la\alpha^+, \bs g\ra \geq c_L\|\bs g - \la\alpha^-, \bs g\ra\ym - \la\alpha^+, \bs g\ra\yp\|_\cE^2.
  \end{equation}

  Let $a^- = \la\alpha^-, \bs g\ra$, $a^+ = \la\alpha^+, \bs g\ra$ and decompose $\bs g = a^-\ym + a^+\yp + \bs k$, so that $\la\alpha^-, \bs k\ra = \la\alpha^+, \bs k\ra = 0$.
  From $\la \cZ, \cY\ra = 0$ we deduce $\la \cZ, k\ra = 0$.
  We have $g = \frac{a^- + a^+}{\nu}\cY + k$ and $\dot g = (-a^- + a^+)\cY + \dot k$, hence
  \begin{align*}
    \frac 12 \la\vD^2E(\bs W)\bs g, \bs g\ra &= \frac 12\big\la \frac {a^- + a^+}{\nu}\cY + k, -(a^- + a^+)\nu \cY + Lk\big\ra \\
    &+ \frac 12\big\la (-a^- + a^+)\cY + \dot k, (-a^- + a^+)\cY + \dot k\big\ra \\
    &= -\frac 12(a^- + a^+)^2\la \cY, \cY\ra -(a^- + a^+)\nu\la \cY, k\ra + \frac 12\la k, Lk\ra \\
    &+ \frac 12(-a^- + a^+)^2 \la \cY, \cY\ra+(-a^- + a^+)\la\cY, \dot k\ra + \frac 12\la \dot k, \dot k\ra \\
    &= -2a^-a^+\la\cY, \cY\ra -2a^-\la\alpha^+, \bs k\ra -2a^+\la\alpha^-, \bs k\ra + \frac 12\big(\la k, Lk\ra + \la \dot k, \dot k\ra\big) \\
    &= -2a^-a^+ + \frac 12\la \vD^2 E(\bs W)\bs k, \bs k \ra.
  \end{align*}
  Invoking Proposition~\ref{prop:coerL} finishes the proof of \eqref{eq:coer-lin-exact}.
\end{proof}
\subsection{Modulation}
\label{ssec:mod}
Recall that $X^s := \dot H^{s+1}\cap \dot H^1$. Let $\bs u^*_0 \in X^s \times H^s$, $T_+ \in \bR$ and let $\bs u^*$ be the solution of \eqref{eq:nlw}
with initial data $\bs u^*(T_+) = \bs u^*_0$. Without loss of generality we will assume that $\frac{N-2}{2} < s \leq 2$.
For fixed $\rho > 0$ we denote
\begin{equation*}
  c^* := \|u^*\|_{L^\infty((T_+ - \rho, T_+) \times B(0, \rho))}.
\end{equation*}
We can assume that $c^* > 0$ (otherwise there is no blow-up, cf. Remark~\ref{rem:Linf}).
Note that because of finite speed of propagation, we can also assume that $\|\bs u^*(t)\|_\cE$ is smaller than any fixed strictly positive constant
and that $\|u^*(t)\|_{L^\infty} \leq 2c^*$ for $t$ close to $T_+$.

Because of a slow decay of $W$, we will introduce compactly supported approximations of $W_\lambda$. Let
\begin{equation}
  \label{eq:R}
  R := (c_0 \cdot c^*)^\frac{1}{-N+2},
\end{equation}
where $c_0 > 0$ is a small universal constant to be chosen later.

We denote
\begin{equation}
  \label{eq:V}
  V(\lambda)(x) :=
  \begin{cases}
    W_\lambda(x) - \zeta(\lambda)\qquad &\text{for }|x| \leq R\sqrt\lambda, \\
    0 \qquad &\text{for }|x| \geq R\sqrt\lambda,
  \end{cases}
\end{equation}
where
\begin{equation}
  \label{eq:ustar-zeta}
  \zeta(\lambda) := W_\lambda(R\sqrt\lambda) = \frac{1}{\lambda^{\frac{N-2}{2}}}\Bigl(1+\frac{R^2}{N(N-2)\lambda}\Bigr)^{-\frac{N-2}{2}}
  = \Bigl(\lambda+\frac{R^2}{N(N-2)}\Bigr)^{-\frac{N-2}{2}}.
\end{equation}
We will also denote
\begin{equation*}
  \bs V(\lambda) := (V(\lambda), 0)\in \cE.
\end{equation*}
Notice that
\begin{equation}
  \label{eq:Vl}
  \partial_\lambda V(\lambda)(x) =
  \begin{cases}
    -(\Lambda W)_{\uln\lambda}(x) - \zeta'(\lambda)\qquad &\text{for }|x| < R\sqrt\lambda, \\
    0 \qquad &\text{for }|x| > R\sqrt\lambda.
  \end{cases}
\end{equation}
\begin{lemma}
  \label{lem:prop-V}
  Let $s > \frac{N-2}{2}$ and $s \geq 1$. The following estimates are true with universal constants:
  \begin{align}
    \|V(\lambda) - W_\lambda\|_{\dot H^1} &\lesssim R^\frac{-N+2}{2}\lambda^\frac{N-2}{4}, \label{eq:V-W} \\
    \|V(\lambda) - W_\lambda\|_{L^\infty} &\lesssim R^{-N+2}, \label{eq:V-Linf} \\
    \|\partial_\lambda V(\lambda) + \Lambda W_\uln\lambda\|_{L^\infty(|x| < R\sqrt\lambda)} &\lesssim R^{-N}, \label{eq:lV-Linf} \\
    \|\partial_\lambda V(\lambda)\|_{L^\frac{2N}{N+2}} &\lesssim R^\frac{6-N}{2}\lambda^\frac{N-2}{4}, \label{eq:lV-H-1} \\
    \|\partial_\lambda V(\lambda)\|_{H^{1-s}} &\ll \lambda^\frac{N-4}{2}\qquad \text{as }\lambda \to 0. \label{eq:lV-H1-s}
  \end{align}
\end{lemma}
\begin{proof}
  To prove \eqref{eq:V-W}, we write
  \begin{equation*}
    \begin{aligned}
      \|V(\lambda) - W_\lambda\|_{\dot H^1}^2 &= \int_{|x| \geq R\sqrt\lambda}|\grad W_\lambda|^2\ud x
      = \int_{|x| \geq R/\sqrt\lambda}\|\grad W\|^2\ud x \\
      &\lesssim \int_{R/\sqrt\lambda}^{+\infty}r^{-2N+2}\cdot r^{N-1}\ud r \sim (R/\sqrt\lambda)^{-N+2}.
    \end{aligned}
  \end{equation*}

  We see that $\zeta(\lambda) \sim R^{-(N-2)}$ and $\zeta'(\lambda) \sim R^{-N}$ when $\lambda$ is small,
  which proves \eqref{eq:V-Linf} and \eqref{eq:lV-Linf}.

  On the support of $\partial_\lambda V(\lambda)$ there holds $|\partial_\lambda V(\lambda)(x)| \lesssim \lambda^\frac{N-4}{2}|x|^{-N+2}$, hence
  \begin{equation*}
    \begin{aligned}
      \|\partial_\lambda V(\lambda)\|_{L^\frac{2N}{N+2}}^\frac{2N}{N+2} &\lesssim \int_0^{R\sqrt\lambda}\lambda^{\frac{N-4}{2}\cdot\frac{2N}{N+2}}r^{(-N+2)\frac{2N}{N+2}}r^{N-1}\ud r \\
      &= \lambda^{\frac{N^2-4N}{N+2}}\int_0^{R\sqrt\lambda}r^\frac{-N^2 + 5N - 2}{N+2}\ud r = R^\frac{N(6-N)}{N+2}\lambda^\frac{N(N-2)}{2(N+2)}.
    \end{aligned}
  \end{equation*}
  This proves \eqref{eq:lV-H-1}.

  We will check \eqref{eq:lV-H1-s} separately in each dimension. For $N = 3$ we have $|\partial_\lambda V(\lambda)(x)| \lesssim \lambda^{-\frac 12}|x|^{-1}$
  and $\||x|^{-1}\|_{L^2(|x| \leq R\sqrt\lambda)} \ll 1$. For $N = 4$ we have $|\partial_\lambda V(\lambda)(x)| \lesssim |x|^{-2}$.
  We suppose $s > 1$, hence there exists $q \in (1, 2)$ such that $L^q \subset H^{1-s}$ and it is easy to check that $\||x|^{-2}\|_{L^q(|x| \leq R\sqrt\lambda)} \ll 1$.
  Finally for $N = 5$ we have $|\partial_\lambda V(\lambda)(x)| \lesssim \lambda^\frac 12 |x|^{-3}$. There exists $q\in \big(1, \frac 53\big)$ such that $L^q \subset H^{1-s}$
  and it is easy to check that $\||x|^{-3}\|_{L^q(|x| \leq R\sqrt\lambda)} \ll 1$.
\end{proof}
Note that $\zeta(\lambda) \sim c_0 c^*$, which means that the cut-off is made at a radius $r = R\sqrt \lambda$ such that $W_{\lambda}(r) \sim c_0 u^*(t, r)$.

For the next lemma we will need the following version of the Implicit Function Theorem.
It is obtained directly from standard proofs of the usual version, see for example \cite[Section 2.2]{chow-hale}.
\begin{lemma}
  \label{lem:implicit}
  Suppose that $X$, $Y$ and $Z$ are Banach spaces, $x_0 \in X$, $y_0 \in Y$, $\rho, \eta > 0$ and $\Phi: B(x_0, \rho) \times B(y_0, \eta) \to Z$
  is continuous in $x$ and continuously differentiable in $y$,
  $\Phi(x_0, y_0) = 0$ and $\vD_y \Phi(x_0, y_0) =: L_0$ has a bounded inverse.
  Suppose that
  \begin{align}
    \|L_0 - \vD_y \Phi(x, y)\|_Z \leq \frac 13 \|L_0^{-1}\|_{\scrL(Z, Y)}^{-1}\qquad &\text{for }\|x-x_0\|_X < \rho, \|y-y_0\|_Y < \eta, \label{eq:implicit-1} \\
    \|\Phi(x, y_0)\|_Z \leq \frac{\eta}{3} \|L_0^{-1}\|_{\scrL(Z, Y)}^{-1}\qquad &\text{for }\|x-x_0\|_X < \rho. \label{eq:implicit-2}
  \end{align}
  Then there exists $y\in C(B(x_0, \rho), B(y_0, \eta))$ such that for $x \in B(x_0, \rho)$, $y(x)$ is the unique solution of the equation $\Phi(x, y(x)) = 0$ in $B(y_0, \eta)$. \qed
\end{lemma}

\begin{lemma}
  \label{lem:loi-mod}
  There exists $\delta_0 > 0$ and $\lambda_0 > 0$ such that for any $0 \leq \delta \leq \delta_0$ and $t_1 < t_2$,
  if $\bs u: (t_1, t_2) \to \cE$ is a solution of \eqref{eq:nlw} satisfying for all $t \in (t_1, t_2)$:
  \begin{equation}
    \label{eq:loi-mod-close}
    \|\bs u(t) - \bs u^*(t) - \bs W_{\wt\lambda(t)}\|_\cE \leq \delta,\qquad 0 < \wt\lambda(t) < \lambda_0,
  \end{equation}
  then there exists a unique function $\lambda(t) \in C^1((t_1, t_2), (0, +\infty))$ such that
  \begin{equation}
    \label{eq:loi-g}
  \bs g(t) := \bs u(t) - \bs u^*(t) - \bs V(\lambda(t))
  \end{equation}
  satisfies for all $t \in (t_1, t_2)$:
  \begin{align}
    \la\cZ_\uln{\lambda(t)}, g(t)\ra &= 0, \label{eq:loi-gorth} \\
    \|\bs g(t)\|_\cE &\lesssim \delta + \wt \lambda(t)^\frac{N-2}{4}, \label{eq:loi-gbound} \\
    |\lambda(t)/\wt \lambda(t) - 1| &\lesssim \delta, \label{eq:loi-lambda-range} \\
    |\lambda'(t)| &\lesssim \|\bs g(t)\|_\cE. \label{eq:loi-mod}
  \end{align}
\end{lemma}
\begin{proof}
  We will first show that for $t_0 \in (t_1, t_2)$ fixed there exists a unique $\lambda(t_0)$ such that \eqref{eq:loi-gorth}, \eqref{eq:loi-gbound} and \eqref{eq:loi-lambda-range} hold at $t = t_0$.
  The proof is standard, see for example \cite[Proposition 1]{MaMe01}.
  
  Denote $\bs v_0 := \bs u(t_0) - \bs u^*(t_0)$ and $\wt l_0 := \log(\wt \lambda(t_0))$ (it will be convenient to consider $\wt \lambda(t_0)$ and $\lambda(t_0)$ in the logarithmic scale).
  We define the following functional:
  \begin{equation*}
    \Phi: \cE\times \bR\to \bR,\qquad \Phi(\bs v; l) := \la \eee^{-l}\cZ_\uln{\eee^l}, v - V(\eee^l)\ra.
  \end{equation*}
  We have
  \begin{equation*}
    \partial_l \Phi(\bs v; l) = -\la \cZ_\uln{\eee^l}, \partial_\lambda V(\eee^l)\ra - \la \eee^{-l}\Lambda_{-1}\cZ_{\uln\eee^l}, v - V(\eee^l)\ra.
  \end{equation*}

  We apply Lemma~\ref{lem:implicit} with $x_0 = \bs V(\wt \lambda(t_0))$ and $y_0 = \wt l_0$.
  It is easily checked that the assumptions hold if $\delta$ is small and $\eta = C\delta$, with a large constant $C$.
  Take $\lambda(t_0) = \eee^{l_0}$, where $l_0$ is the solution of $\Phi(\bs v_0; l_0) = 0$ given by Lemma~\ref{lem:implicit}.
  Directly from the definition of $\Phi$ we obtain \eqref{eq:loi-gorth}.
  The inequality $|l_0 - \wt l_0| \leq \eta = C\delta$ is equivalent to \eqref{eq:loi-lambda-range}, which in turn implies
  \begin{equation}
    \label{eq:var-norm-sol}
    \|W_{\wt \lambda(t_0)} - W_{\lambda(t_0)}\|_{\dot H^1} \lesssim \delta.
  \end{equation}
  From the definition of $\bs g$ and \eqref{eq:loi-mod-close} we have
  \begin{equation*}
    \|\bs g\| \leq \delta + \|W_{\wt \lambda(t_0)} - W_{\lambda(t_0)}\|_{\dot H^1} + \|W_{\lambda(t_0)} - V(\lambda(t_0))\|_{\dot H^1},
  \end{equation*}
  so \eqref{eq:loi-gbound} follows from \eqref{eq:var-norm-sol} and \eqref{eq:V-W}.

  For each $t_0 \in (t_1, t_2)$ we have defined $\lambda(t_0)$. It remains to show that $\lambda(t)$ is a $C^1$ function and that \eqref{eq:loi-mod} holds.
  One way is to use a regularization procedure as in \cite{MaMe01}. Here we give a different argument, which might be simpler in some cases.

  Take $t_0 \in (t_1, t_2)$ and let $l_0 := \log(\lambda(t_0))$.
  Denote $\bs v(t) := \bs u(t) - \bs u^*(t)$ and define $l: (t_0-\varepsilon, t_0+\varepsilon) \to \bR$ as the solution of the differential equation
  \begin{equation}
    \label{eq:lambda-der}
    l'(t) = -(\partial_l \Phi)^{-1}(\vD_{\bs v}\Phi)\partial_t \bs v(t)
  \end{equation}
  with the initial condition $l(t_0) = l_0$.
  Notice that $\vD_{\bs v}\Phi$ is a continuous functional on $\cF$, so we can apply it to $\partial_t \bs v(t)$.

  Using the chain rule we get $\dd t \Phi(\bs v(t); l(t)) = 0$ for $t \in (t_0 - \varepsilon, t_0 + \varepsilon)$.
  By continuity, $|l(t) - l_0| <\eta = C\delta$ in some neighbourhood of $t = t_0$.
  Hence, by the uniqueness part of Lemma~\ref{lem:implicit}, we get $l(t) = \log \lambda(t)$ in some neighbourhood of $t = t_0$.
  In particular, $\lambda(t)$ is of class $C^1$ in some neighbourhood of $t_0$.

  From \eqref{eq:loi-g} we obtain the following differential equation for the error term $\bs g$:
  \begin{equation}
    \label{eq:loi-dtg}
    \partial_t \bs g = J\circ(\vD E(\bs V(\lambda) + \bs u^* + \bs g) - \vD E(\bs u^*)) - \lambda'\partial_\lambda \bs V(\lambda),
  \end{equation}
  which can also be written in the expanded form
  \begin{equation}
    \label{eq:loi-dtg-2}
    \bigg\{
    \begin{aligned}
      \partial_t g &= \dot g - \lambda'\partial_\lambda V(\lambda), \\
      \partial_t \dot g &= \Delta g + \big(f(u^* + V(\lambda) + g) - f(u^*) - f(V(\lambda))\big) + \big(\Delta V(\lambda) + f(V(\lambda))\big).
    \end{aligned}
  \end{equation}
  Differentiating \eqref{eq:loi-gorth} and using the first equation in \eqref{eq:loi-dtg-2} we get
  \begin{equation*}
    \begin{aligned}
    0 &= \dd t\la \cZ_\uln\lambda, g\ra = -\frac{\lambda'}{\lambda}\la \Lambda_0\cZ_\uln\lambda, g\ra + \la \cZ_\uln\lambda, \dot g - \lambda'\partial_\lambda V(\lambda)\ra \\
    &= \lambda'\big(\la \cZ, \Lambda W\ra - \la \cZ_\uln\lambda, \Lambda W_\uln\lambda + \partial_\lambda V(\lambda)\ra - \big\la\frac{1}{\lambda}\Lambda_0\cZ_\uln\lambda, g\big\ra\big)
    + \la \cZ_\uln\lambda, \dot g\ra.
  \end{aligned}
  \end{equation*}
  We assumed that $\la \cZ, \Lambda W\ra > 0$. When $\|\bs g\|_\cE$ and $\lambda$ are small enough, then
  \begin{equation*}
    \big|\la \cZ_\uln\lambda, \Lambda W_\uln\lambda + \partial_\lambda V(\lambda)\ra + \big\la\frac{1}{\lambda}\Lambda_0\cZ_\uln\lambda, g\big\ra\big| \leq \frac 12 \la \cZ, \Lambda W\ra
  \end{equation*}
  (we use \eqref{eq:lV-Linf} in order to estimate the first term). This proves \eqref{eq:loi-mod}.
\end{proof}

If $\bs u(t)$ is a solution of \eqref{eq:nlw} satisfying \eqref{eq:loi}, then there exists $t_0$ such that \eqref{eq:loi-mod-close}
holds for $t \in [t_0, T_+)$. It follows from \eqref{eq:loi-lambda-range} that, while proving Theorem~\ref{thm:loi},
without loss of generality we can assume that $\lambda(t)$ is the function given by Lemma~\ref{lem:loi-mod}.
From \eqref{eq:loi-gbound} we obtain that $\|\bs g\|_\cE \to 0$ as $t \to T_+$, which is the only information about $\bs g$ used in the sequel.
The precise form of the right hand side of \eqref{eq:loi-gbound} has no importance.
We will prove that \eqref{eq:loi-borne} holds on some interval $[t_0, T_+)$ with $t_0 < T_+$, with no information about the length of this interval.
Each time we state something for $t \in [t_0, T_+)$ it should be understood that $t_0$ is sufficiently close to $T_+$.

In the rest of this paper $\lambda(t)$ always stands for the modulation parameter obtained in Lemma~\ref{lem:loi-mod} and $\bs g(t)$ is the function defined by \eqref{eq:loi-g}.
We introduce the following notation for the joint size of the error and the interaction:
$$n(\bs g, \lambda) := \sqrt{\|\bs g\|_\cE^2 + c^*\lambda^\frac{N-2}{2}}.$$

We will now analyze the stable and unstable directions of the linearized flow. The stable coefficient $a^-(t)$ and the unstable coefficient $a^+(t)$ are defined as follows:
\begin{equation*}
    a^-(t) := \la \alpha_{\lambda(t)}^-, \bs g(t)\ra, \qquad a^+(t) := \la \alpha_{\lambda(t)}^+, \bs g(t)\ra.
\end{equation*}
Note that $|a^-(t)| \lesssim \|\bs g\|_\cE$ and $|a^+(t)| \lesssim \|\bs g\|_\cE$.
\begin{lemma}
  \label{lem:amp}
  The functions $a^-(t)$ and $a^+(t)$ satisfy
  \begin{align}
    \big|\dd t a^-(t) + \frac{\nu}{\lambda(t)}a^-(t)\big| &\lesssim \frac{1}{\lambda(t)}n(\bs g(t), \lambda(t))^2,\label{eq:loi-alpham}\\
    \big|\dd t a^+(t) - \frac{\nu}{\lambda(t)}a^+(t)\big| &\lesssim \frac{1}{\lambda(t)}n(\bs g(t), \lambda(t))^2.\label{eq:loi-alphap}
  \end{align}
\end{lemma}
\begin{proof}
  We will only prove \eqref{eq:loi-alphap}; the other estimate can be shown analogously.

  Let us rewrite equation \eqref{eq:loi-dtg-2} in the following manner:
  \begin{equation*}
    \partial_t \bs g = J\circ \vD^2 E(\bs W_\lambda)\bs g + \bs h,
  \end{equation*}
  where
  \begin{equation*}
    \bs h = \begin{pmatrix}h \\ \dot h\end{pmatrix} = \begin{pmatrix}
      -\lambda'\partial_\lambda V(\lambda), \\
      \big(f(u^* + V(\lambda) + g) - f(u^*) - f(V(\lambda)) - f'(W_\lambda)g\big) + \big(\Delta V(\lambda) + f(V(\lambda))\big)\end{pmatrix}.
  \end{equation*}
  Using \eqref{eq:eigencovect} we get
  \begin{equation*}
    \begin{aligned}
      \dd t a^-(t) + \frac{\nu}{\lambda}a^-(t) &= \dd t\la \alpha_\lambda^-, \bs g\ra + \frac{\nu}{\lambda}\la \alpha_\lambda^-, \bs g\ra \\ &= -\frac{\lambda'}{\lambda}\la \Lambda_{\cE^*}\alpha_\lambda^-, \bs g\ra + \la \alpha_\lambda^-, J\circ \vD^2 E(\bs W_\lambda)\bs g\ra + \frac{\nu}{\lambda}\la \alpha_\lambda^-, \bs g\ra + \la\alpha_\lambda^-, \bs h\ra \\ &= -\frac{\lambda'}{\lambda}\la \Lambda_{\cE^*}\alpha_\lambda^-, \bs g\ra+ \la\alpha_\lambda^-, \bs h\ra.
  \end{aligned}
  \end{equation*}
  The first term is negligible due to \eqref{eq:loi-mod}. In order to bound the second term it suffices to check the following inequalities:
  \begin{align}
    \big|\big\la\cY_\uln\lambda, \partial_\lambda V(\lambda)\big\ra\big| &\lesssim n(\bs g, \lambda)^2, \label{eq:alpha-err-1} \\
    \big|\big\la\cY_\lambda, \big(\Delta V(\lambda) + f(V(\lambda))\big)\big\ra\big| &\lesssim n(\bs g, \lambda)^2, \label{eq:alpha-err-2} \\
    \big|\big\la\cY_\lambda, \big(f(u^* + V(\lambda) + g) - f(u^*) - f(V(\lambda)) - f'(W_\lambda)g\big) \big\ra\big| &\lesssim n(\bs g, \lambda)^2. \label{eq:alpha-err-3}
  \end{align}
  The first inequality follows from \eqref{eq:lV-Linf} and \eqref{eq:YLW}, since the region $|x| \geq R\sqrt\lambda$ is negligible due to exponential decay of $\cY$.

  Notice that $|f(W_\lambda) - f(V(\lambda))| \lesssim f'(W_\lambda)|\cdot |W_\lambda - V(\lambda)| \lesssim f'(W_\lambda)c_0c^*$,
  where the last inequality follows from \eqref{eq:V-W} and \eqref{eq:R}. Together with the fact that $\Delta(W_\lambda) + f(W_\lambda) = 0$ this implies
  \begin{equation*}
    \begin{aligned}
      \big|\la \cY_\lambda, \big(\Delta V(\lambda) + f(V(\lambda))\big)\big\ra\big| &\lesssim \big|\big\la \cY_\lambda, \Delta\big(W_\lambda - V(\lambda)\big)\big\ra\big| + \big|\big\la \cY_\lambda, f(W_\lambda) - f(V(\lambda))\big\ra\big| \\
      &\lesssim \big(\|\Delta Y_\lambda\|_{L^1} + \|f'(W_\lambda)\cY_\lambda\|_{L^1}\big)c_0 c^* \lesssim c^*\lambda^\frac{N-2}{2},
    \end{aligned}
  \end{equation*}
  which proves \eqref{eq:alpha-err-2}.

  We will check \eqref{eq:alpha-err-3} in three small steps. As before, we do not have to worry about the region $|x| \geq R\sqrt\lambda$ thanks to the fast decay of $\cY$.
  First, we have a pointwise bound 
  \begin{equation}
    \label{eq:pointwise-ustar-V}
      |f(u^* + V(\lambda)) - f(u^*) - f(V(\lambda))| \lesssim f'(W_\lambda)\cdot c^* + f(c^*),
  \end{equation}
  which implies
  \begin{equation}
    \label{eq:alpha-err-4}
    \big|\big\la \cY_\lambda, f(u^* + V(\lambda)) - f(u^*) - f(V(\lambda))\big\ra\big| \lesssim n(\bs g, \lambda)^2.
  \end{equation}
  Next, we have
  \begin{equation}
    \label{eq:pointwise-ustar-V-g}
  |f(u^* + V(\lambda) + g) - f(u^* + V(\lambda)) - f'(u^* + V(\lambda))g| \lesssim |f''(u^* + V(\lambda))|\cdot|g|^2 + f(|g|),
\end{equation}
  which implies
  \begin{equation}
    \label{eq:alpha-err-5}
    \big|\big\la \cY_\lambda, f(u^* + V(\lambda) + g) - f(u^* + V(\lambda)) - f'(u^* + V(\lambda))g\big\ra\big| \lesssim n(\bs g, \lambda)^2.
  \end{equation}
  Finally, $ |f'(V(\lambda) + u^*) - f'(W_\lambda)| \lesssim (|f''(W_\lambda)| + |f''(V(\lambda) + u^* - W_\lambda)|)\cdot|V(\lambda) + u^* - W_\lambda| \lesssim |f''(W_\lambda)|c^*$.
  Using H\"older and the fact that $\|\cY_\lambda\cdot f''(W_\lambda)\|_{L^\frac{2N}{N+2}} \lesssim \lambda^\frac{N-2}{2}$ this implies
  \begin{equation}
    \label{eq:alpha-err-6}
    \big|\big\la \cY_\lambda, \big(f'(u^* + V(\lambda)) - f'(W_\lambda)\big)g\big\ra\big| \ll n(\bs g, \lambda)^2.
  \end{equation}
  Now \eqref{eq:alpha-err-3} follows from \eqref{eq:alpha-err-4}, \eqref{eq:alpha-err-5} and \eqref{eq:alpha-err-6} and the triangle inequality.
\end{proof}

\subsection{Coercivity}
\label{ssec:coer}
By the conservation of energy, for all $t \in [t_0, T_+)$ there holds
\begin{equation}
  \label{eq:conservation}
  E(\bs V(\lambda) + \bs u^* + \bs g) = E(\bs W) + E(\bs u^*).
\end{equation}
On the other hand, using the pointwise inequality
\begin{equation*}
    |F(k+l) - F(k) - f(k)l - \frac 12 f'(k)l^2| \lesssim |f''(k)||l^3| + |F(l)|,\qquad \forall k, l \in \bR
  \end{equation*}
  we deduce that
 \begin{equation*}
    \begin{aligned}
      E(\bs V(\lambda) + \bs u^* + \bs g) &= E(\bs V(\lambda) + \bs u^*) + \la \vD E(\bs V(\lambda) + \bs u^*), \bs g\ra \\
      &+ \frac 12\la \vD^2 E(\bs V(\lambda) + \bs u^*)\bs g, \bs g\ra + O(\|\bs g\|_\cE^3).
    \end{aligned}
  \end{equation*}
  Using \eqref{eq:conservation} we obtain
  \begin{equation}
    \label{eq:loi-energy}
    \begin{aligned}
      \big(E(\bs V(\lambda) + \bs u^*)-E(\bs W) - E(\bs u^*)\big) &+ \la \vD E(\bs V(\lambda) + \bs u^*), \bs g\ra \\
      &+ \frac{1}{2}\la \vD^2 E(\bs V(\lambda) + \bs u^*)\bs g, \bs g\ra = O(\|\bs g\|_\cE^3).
    \end{aligned}
  \end{equation}
We start by computing the size of the first term on the left hand side.
\begin{lemma}
  \label{lem:loi-coer}
  For $T_+ - t$ small there holds
  \begin{equation}
    \label{eq:inter-bound}
    |E(\bs V(\lambda) + \bs u^*) - E(\bs W) - E(\bs u^*)| \lesssim c^*\lambda^\frac{N-2}{2}.
  \end{equation}
  In addition, if $u^*(0) < 0$, then
  \begin{equation}
    \label{eq:inter-positive}
    E(\bs V(\lambda) + \bs u^*) - E(\bs W) - E(\bs u^*) \gtrsim c^*\lambda^\frac{N-2}{2}.
  \end{equation}
\end{lemma}
\begin{proof}
  Integrating by parts we obtain
  \begin{equation*}
    \begin{aligned}
      \int\grad V(\lambda)\cdot\grad u^*\ud x &= \int_{B(0, R\sqrt\lambda)}\grad(W_\lambda)\cdot \grad u^*\ud x \\
      &= -\int_{B(0, R\sqrt\lambda)}\Delta(W_\lambda)\cdot u^*\ud x + \int_{S(0, R\sqrt\lambda)}\partial_r(W_\lambda)\cdot u^*\ud\sigma \\
      &= \int_{B(0, R\sqrt\lambda)}f(W_\lambda)\cdot u^*\ud x + \int_{S(0, R\sqrt\lambda)}\partial_r(W_\lambda)\cdot u^*\ud\sigma.
    \end{aligned}
  \end{equation*}
  Developping the energy gives
\begin{equation}
  \label{eq:en-expansion}
  \begin{aligned}
  E(\bs V(\lambda) + \bs u^*) - E(\bs W) - E(\bs u^*) &= \int \grad V(\lambda)\cdot \grad u^*\ud x + \frac 12\int|\grad V(\lambda)|^2 - |\grad (W_\lambda)|^2 \ud x \\
  &-\int F(V(\lambda)+u^*)-F(W_\lambda) - F(u^*)\ud x \\
  &= \int_{S(0, R\sqrt\lambda)}\partial_r(W_\lambda)\cdot u^*\ud\sigma + \frac 12 \int |\grad V(\lambda)|^2 - |\grad W_\lambda|^2 \ud x \\
  &- \int F(V(\lambda)+u^*)-F(W_\lambda) - F(u^*) - f(V(\lambda))\cdot u^*\ud x \\ &+ \int_{B(0, R\sqrt\lambda)}\big(f(W_\lambda)-f(V(\lambda))\big)\cdot u^*\ud x.
\end{aligned}
\end{equation}
We will show that all the terms on the right hand side except for the first one are $\lesssim c_0 c^* \lambda^\frac{N-2}{2}$,
where $c_0$ is the small constant in \eqref{eq:R}.

The fact that $\int |\grad V(\lambda)|^2 - |\grad W_\lambda|^2 \ud x \lesssim c_0 c^* \lambda^\frac{N-2}{2} = R^{-N + 2}\lambda^\frac{N-2}{2}$
follows directly from the proof of \eqref{eq:V-W}.

We will now show that
  \begin{equation*}
    \int |F(V(\lambda) + u^*) - F(V(\lambda)) - F(u^*) - f(V(\lambda))u^*|\ud x \ll \lambda^{\frac{N-2}{2}}.
  \end{equation*}
  To this end, notice first that the integrand equals $0$ for $|x| \geq R\sqrt\lambda$.
  In the region $|x| \leq R\sqrt\lambda$ we use the pointwise estimate
  \begin{equation*}
    |F(V(\lambda)+u^*)-F(V(\lambda))-F(u^*)-f(V(\lambda))u^*| \lesssim f'(V(\lambda))|u^*|^2 + F(u^*).
  \end{equation*}
  The term $F(u^*)$ can be neglected (it is bounded in $L^\infty$,
  so its contribution is at most $\lambda^{\frac{N}{2}} \ll \lambda^{\frac{N-2}{2}}$).
  As for the first term, it is easily checked that
  \begin{equation}
    \label{eq:fprim}
    \int_{|x| \leq R\sqrt\lambda}f'(W_\lambda)\ud x = \lambda^{N-2}\int_{|x|\leq R/\sqrt\lambda}f'(W)\ud x \ll \lambda^{\frac{N-2}{2}}.
  \end{equation}

  Next, we show that if $R$ is large enough, then
  \begin{equation*}
    \int|F(W_\lambda) - F(V(\lambda))|\ud x \lesssim c_0 c^* \lambda^{\frac{N-2}{2}}.
  \end{equation*}
  In the region $|x| \geq R\sqrt\lambda$ from \eqref{eq:V-W} and Sobolev embedding we obtain that the contribution is at most $\lambda^{\frac N2} \ll \lambda^{\frac{N-2}{2}}$.
  In the region $|x| \leq R\sqrt\lambda$ we use the bound
  \begin{equation*}
    |F(W_\lambda) - F(V(\lambda))| \lesssim \zeta(\lambda)\cdot|f(W_\lambda)| + F(\zeta(\lambda)).
  \end{equation*}
  The second term is in $L^\infty$, so its integral is at most $O(\lambda^{\frac N2}) \ll \lambda^{\frac{N-2}{2}}$.
  As for the first term, it is easily seen that $\int |f(W_\lambda)|\ud x \lesssim \lambda^{\frac{N-2}{2}}$,
  and we get the conclusion if we recall that $\zeta(\lambda)\sim c_0 c^*$.

  Finally, from \eqref{eq:fprim} and the pointwise bound
  $|f(V(\lambda)) - f(W_\lambda)| \lesssim |\zeta(\lambda)f'(W_\lambda)| + |f(\zeta(\lambda))|$ it follows that
  \begin{equation*}
    \int_{B(0, R\sqrt\lambda)}|f(V(\lambda)) - f(W_\lambda)|\cdot|u^*|\ud x \ll \lambda^{\frac{N-2}{2}}.
  \end{equation*}

  Now consider the first term on the right hand side of \eqref{eq:en-expansion}.
  We have $\partial_r(W_\lambda)(R\sqrt\lambda) \sim -\lambda^{\frac{N-2}{2}}(R\sqrt\lambda)^{-N+1}$ and $|u^*| \leq c^*$ near the origin,
  so we get
  \begin{equation*}
    \Big|\int_{S(0, R\sqrt\lambda)}\partial_r(W_\lambda)\cdot u^*\ud\sigma\Big| \lesssim c^*\lambda^\frac{N-2}{2}.             
  \end{equation*}
In the case $u^*_0(0) < 0$, by continuity if in the definition of $c^*$ we choose $\rho$ small enough, then $u^*(t, x) \leq -\frac 12 c^*$ for $(t, x) \in [t_0, T_+) \times B(0, \rho)$.
  In particular,
  \begin{equation*}
    \int_{S(0, R\sqrt\lambda)}\partial_r(W_\lambda)\cdot u^*\ud\sigma \gtrsim c^*\lambda^\frac{N-2}{2},             
  \end{equation*}
  where the constant in this estimate is independent of $c_0$. The conclusion follows from \eqref{eq:en-expansion} if $c_0$ is chosen small enough.
\end{proof}

We will focus at present on the second term on the left hand side of \eqref{eq:loi-energy}. In Lemma~\ref{lem:loi-gamma-reg} we treat
the simpler case $\bs u^*_0 \in X^2 \times H^2$ and in Lemma~\ref{lem:loi-gamma} we prove a weaker estimate in the case
$\bs u^* \in X^s \times H^s$, $s > \frac{N-2}{2}$, $s \geq 1$.
\begin{lemma}
  \label{lem:loi-gamma-reg}
  Suppose that $\bs u^*_0 \in X^2 \times H^2$. Then for $t \in [t_0, T_+)$ there holds
    \begin{equation*}
      |\la \vD E(\bs V(\lambda(t)) + \bs u^*(t)), \bs g(t)\ra| \lesssim \sqrt{c_0}\cdot \sup_{t\leq \tau < T_+}n(\bs g(\tau), \lambda(\tau))^2,
    \end{equation*}
    where $c_0$ is the small constant in \eqref{eq:R}.
\end{lemma}
\begin{proof}
  The proof has two steps. First we will show that
  \begin{equation}
    \label{eq:pas-de-bulle}
    |\la \vD E(\bs V(\lambda(t)) + \bs u^*(t)) - \vD E(\bs u^*(t)), \bs g(t)\ra| \lesssim \sqrt{c_0}\cdot n(\bs g(t), \lambda(t))^2
  \end{equation}
  and then we will check that
  \begin{equation}
    \label{eq:magique}
    \big|\dd t\la \vD E(\bs u^*(t)), \bs g\ra\big| \lesssim_R n(\bs g(t), \lambda(t))^2.
  \end{equation}
  Clearly, integrating \eqref{eq:magique} and using \eqref{eq:pas-de-bulle}, we obtain the conclusion for $t_0$ sufficiently close to $T_+$.
  Note that the constant in \eqref{eq:magique} is allowed to depend on $R$ (because $T_+ - t_0$ can also be chosen depending on $R$).

  In order to prove \eqref{eq:pas-de-bulle}, we begin by verifying that
  \begin{equation}
    \label{eq:ustar-final1}
    \big|\la \vD E(\bs V(\lambda) + \bs u^*), \bs g\ra - \la\vD E(\bs V(\lambda)), \bs g\ra - \la\vD E(\bs u^*), \bs g\ra\big| \ll n(\bs g, \lambda)^2.
  \end{equation}
  This is equivalent to
  \begin{equation*}
    \int|f(V(\lambda) + u^*) - f(V(\lambda)) - f(u^*)|\cdot|g|\ud x \ll n(\bs g, \lambda)^2.
  \end{equation*}
  By H\"older and Sobolev inequalities, it suffices to show that
  \[
    \|f(V(\lambda) + u^*) - f(V(\lambda)) - f(u^*)\|_{L^{\frac{2N}{N+2}}} \ll \lambda^{\frac{N-2}{4}}.
  \]
  Using \eqref{eq:pointwise-ustar-V} we obtain easily that the left hand side is $\lesssim \lambda^\frac{N-2}{2}$.

  Recall that $R^{-N + 2} = c_0 c^*$, hence \eqref{eq:V-W} gives $\|W_\lambda - V(\lambda)\|_{\dot H^1} \lesssim \sqrt{c_0 c^*}$.
  Using $\Delta W_\lambda + f(W_\lambda) = 0$ and the pointwise bound $|f(W_\lambda) - f(V(\lambda))| \lesssim f'(W_\lambda)\cdot |W_\lambda - V(\lambda)|$ one gets
  $$
  \|\Delta V(\lambda) + f(V(\lambda))\|_{\dot H^{-1}} \lesssim \|\Delta(W_\lambda - V(\lambda))\|_{\dot H^{-1}} + \|f(W_\lambda) - f(V(\lambda))\|_{L^\frac{2N}{N+2}} \lesssim \sqrt{c_0c^*},
  $$
  hence
  \begin{equation}
    \label{eq:ustar-final2}
    |\la \vD E(\bs V(\lambda)), \bs g\ra| \lesssim \sqrt{c_0}\cdot n(\bs g, \lambda)^2.
  \end{equation}
  Estimate \eqref{eq:pas-de-bulle} follows from \eqref{eq:ustar-final1} and \eqref{eq:ustar-final2}.
  Notice that until now the assumption $\bs u_0^* \in X^2 \times H^2$ has not been used, thus \eqref{eq:pas-de-bulle} holds also in the case $\bs u_0^* \in X^s\times H^s$, $s > \frac{N-2}{2}$.

  We move on to the proof of \eqref{eq:magique}. Until the end of this proof all the constants are allowed to depend on $R$.
  From \eqref{eq:loi-dtg} we get
  \begin{equation*}
    \begin{aligned}
      \dd t\la \vD E(\bs u^*), \bs g\ra &= \la \vD^2 E(\bs u^*)\partial_t \bs u^*, \bs g\ra + \big\la \vD E(\bs u^*), J\circ\big(\vD E(\bs V(\lambda) + \bs u^* + \bs g) - \vD E(\bs u^*)\big)-\lambda'\partial_\lambda \bs V(\lambda)\big\ra.
    \end{aligned}
  \end{equation*}
  Notice that
  \begin{equation*}
    \la \vD^2 E(\bs u^*)\partial_t \bs u^*, \bs g\ra = -\la \vD E(\bs u^*), J\circ \vD^2 E(\bs u^*)\bs g\ra,
  \end{equation*}
  hence it suffices to verify that
  \begin{equation*}
    \big|\big\la \vD E(\bs u^*), J\circ\big(\vD E(\bs V(\lambda) + \bs u^* + \bs g)-\vD E(\bs u^*) - \vD^2 E(\bs u^*)\bs g\big) - \lambda'\partial_\lambda \bs V(\lambda)\big\ra\big| \lesssim n(\bs g, \lambda)^2.
  \end{equation*}
  Considering separately the first and the second component, cf. \eqref{eq:loi-dtg-2}, we obtain that it is sufficient to verify the following bounds:
  \begin{align}
    \label{eq:gamma-err-1}
    |\la \Delta u^* + f(u^*), \lambda'\partial_\lambda V(\lambda)\ra| &\lesssim n(\bs g, \lambda)^2, \\
    \label{eq:gamma-err-2}
    |\la \dot u^*, f(V(\lambda) + u^* + g) - f(V_\lambda) - f(u^*) - f'(u^*)g\ra| &\lesssim n(\bs g, \lambda)^2, \\
    \label{eq:gamma-err-3}
    |\la \dot u^*, \Delta V(\lambda) + f(V(\lambda))\ra| &\lesssim n(\bs g, \lambda)^2.
  \end{align}

  We know from Appendix~\ref{sec:cauchy} that $u^*(t)$ is bounded in $X^2$,
  hence $\Delta u^* + f(u^*)$ is bounded in $L^\frac{2N}{N-2}$ by the Sobolev embedding.
  From \eqref{eq:lV-H-1} and H\"older inequality it follows that
  $$|\la \Delta u^* + f(u^*), \partial_\lambda V(\lambda)\ra| \lesssim \lambda^\frac{N-2}{4},$$
  and \eqref{eq:gamma-err-1} follows from \eqref{eq:loi-mod}.

  Since $\dot u^*(t)$ is bounded in $L^\frac{2N}{N-2}$, in order to prove \eqref{eq:gamma-err-2} it suffices (by H\"older) to check that
  \begin{equation}
    \label{eq:dest-inter-0}
    \|f(V(\lambda) + u^* + g) - f(V(\lambda)) - f(u^*) - f'(V(\lambda) + u^*)g\|_{L^\frac{2N}{N+2}} \lesssim n(\bs g, \lambda)^2
  \end{equation}
  and
  \begin{equation}
    \label{eq:dest-inter-1}
    \|\dot u^*\cdot(f'(V(\lambda) + u^*) - f'(u^*))\|_{L^\frac{2N}{N+2}} \lesssim \lambda^\frac{N-2}{4}.
  \end{equation}
  We first prove \eqref{eq:dest-inter-1}.
  For $|x| \geq R\sqrt\lambda$ the integrand equals 0, and in the region $|x| \leq R\sqrt\lambda$ there holds $|f'(V(\lambda))| + |f'(u^*)| \lesssim f'(W_\lambda)$.
  \begin{itemize}
    \item For $N = 3$ $\dot u^* \in H^2 \subset L^\infty$ and $\|f'(W_\lambda)\|_{L^\frac 65} \lesssim \lambda^\frac 12$.
    \item For $N = 4$ $\dot u^* \in H^2 \subset L^{12}$ and $\|f'(W_\lambda)\|_{L^\frac 32} \lesssim \lambda^\frac 23$.
    \item For $N = 5$ $\dot u^* \in H^2 \subset L^{10}$ and $\|f'(W_\lambda)\|_{L^\frac 53} \lesssim \lambda$.
  \end{itemize}
  In all three cases \eqref{eq:dest-inter-1} follows from H\"older inequality.

  By a pointwise bound we have
  \begin{equation*}
    \|f(V(\lambda) + u^*) - f(V(\lambda)) - f(u^*)\|_{L^\frac{2N}{N+2}} \lesssim \|u^*\cdot f'(V(\lambda))\|_{L^\frac{2N}{N+2}} + \|f'(u^*)\cdot V(\lambda)\|_{L^\frac{2N}{N+2}}.
  \end{equation*}
  It is easy to check that
  $$\|f'(V(\lambda))\|_{L^\frac{2N}{N+2}} \leq \|f'(W_\lambda)\|_{L^\frac{2N}{N+2}} \lesssim \lambda^\frac{N-2}{2}.$$
  Together with \eqref{eq:lV-H-1} this yields
  \begin{equation}
    \label{eq:f_V+u}
    \|f(V(\lambda) + u^*) - f(V(\lambda)) - f(u^*)\|_{L^\frac{2N}{N+2}} \lesssim \lambda^\frac{N-2}{2},
  \end{equation}
  and \eqref{eq:dest-inter-0} follows from \eqref{eq:pointwise-ustar-V-g} and the H\"older inequality.

  In order to prove \eqref{eq:gamma-err-3}, we write:
  \begin{equation}
    \label{eq:gamma-err-4}
    |\la \dot u^*, \Delta V(\lambda) + f(V(\lambda))\ra| \leq |\la \dot u^*, \Delta(V(\lambda) - W_\lambda)\ra| + |\la \dot u^*, f(V(\lambda) - f(W_\lambda)\ra|.
  \end{equation}
  Consider the first term of \eqref{eq:gamma-err-4}. Integrating twice by parts we find
  \begin{equation*}
    \begin{aligned}
      \int\dot u^*\cdot \Delta(V(\lambda) - W_\lambda)\ud x &= \int_{|x| \geq R\sqrt\lambda}\grad \dot u^*\cdot \grad (W_\lambda)\ud x \\
      &= \int_{S(0, R\sqrt\lambda)}\dot u^*\cdot \partial_r(W_\lambda)\ud\sigma - \int_{|x| \geq R\sqrt\lambda}\dot u^*\cdot \Delta(W_\lambda)\ud x.
    \end{aligned}
  \end{equation*}
  As for the first term, recall that $|\partial_r(W_\lambda(R\sqrt\lambda))| \lesssim \lambda^\frac{N-2}{2}$,
  so it suffices to notice that by the Trace Theorem $\int |\dot u^*|\ud \sigma \ll 1$ for $\lambda \ll 1$.
  In order to bound the second term, we compute
  \begin{equation*}
    \|f(W_\lambda)\|_{L^\frac{2N}{N+2}(|x|\geq R\sqrt\lambda)} = \|f(W)\|_{L^\frac{2N}{N+2}(|x|\geq R/\sqrt\lambda)} \sim \lambda^\frac{N+2}{4} \ll \lambda^\frac{N-2}{2},
  \end{equation*}
  and use H\"older.

  Consider the second term of \eqref{eq:gamma-err-4}. From \eqref{eq:V-Linf} we have $|f(V(\lambda)) - f(W_\lambda)| \lesssim f'(W_\lambda)$, hence:
  \begin{equation*}
    \Big|\int\dot u^*\cdot \big(f(V(\lambda)) - f(W_\lambda)\big)\ud x\Big| \lesssim \int |\dot u^*|\cdot f'(W_\lambda)\ud x,
  \end{equation*}
  and the required bound follows from H\"older and the fact that $\|f'(W_\lambda)\|_{L^\frac{2N}{N+2}} \lesssim \lambda^\frac{N-2}{2}$.
\end{proof}

\begin{lemma}
  \label{lem:loi-gamma}
  Suppose that $\bs u^*_0 \in X^s \times H^s$, $s > \frac{N-2}{2}$ and $s \geq 1$. There exists a decomposition
  \begin{equation*}
    \la \vD E(\bs V(\lambda(t)) + \bs u^*(t)), \bs g(t)\ra  = b_1(t) + b_2(t)
  \end{equation*}
  such that for $t \in [t_0, T_+)$ there holds:
  \begin{align}
    |b_1'(t)| &\ll \lambda(t)^\frac{N-4}{2}\|\bs g\|_\cE, \label{eq:b1p} \\
    |b_2(t)| &\lesssim \sqrt{c_0}\cdot \sup_{t \leq \tau < T_+}n(\bs g(\tau), \lambda(\tau))^2. \label{eq:b2}
  \end{align}
\end{lemma}
\begin{proof}
  We take
  \begin{equation*}
    \begin{aligned}
      b_1(t) &:= \la \vD E(\bs u^*(t)), \bs g(t)\ra, \\
      b_2(t) &:= \la \vD E(\bs V(\lambda(t)) + \bs u^*(t)) - \vD E(\bs u^*(t)), \bs g(t)\ra.
    \end{aligned}
  \end{equation*}
  Estimate \eqref{eq:b2} is exactly \eqref{eq:pas-de-bulle}.

  Repeating the computation in the proof of Lemma~\ref{lem:loi-gamma-reg}, we see that we need to check inequalities \eqref{eq:gamma-err-1}, \eqref{eq:gamma-err-2} and \eqref{eq:gamma-err-3},
  with ``~$\lesssim n(\bs g, \lambda)^2$~'' replaced by ``~$\ll \lambda^\frac{N-4}{2}\|\bs g\|$~''.

  We know that $\Delta u^*$ is bounded in $H^{s-1}$, hence from \eqref{eq:lV-H1-s} we obtain $|\la \Delta u^*, \partial_\lambda V(\lambda)\ra| \ll \lambda^\frac{N-4}{2}$.
  Since $\|f(u^*)\|_{L^\frac{2N}{N-2}}$ is bounded and $\frac{N-2}{4} > \frac{N-4}{2}$, from \eqref{eq:lV-H-1} we get $|\la f(u^*), \partial_\lambda V(\lambda)\ra| \ll \lambda^\frac{N-4}{2}$.
  Using \eqref{eq:loi-mod}, it follows that
  \begin{equation}
    \label{eq:gamma-err-21}
    |\la \Delta u^* + f(u^*), \lambda'\partial_\lambda V(\lambda)\ra| \ll \lambda^\frac{N-4}{2}\|\bs g\|.
  \end{equation}
  The proof of \eqref{eq:gamma-err-2} applies almost without changes, but instead of \eqref{eq:dest-inter-1} we need to check that $\|\dot u^*\cdot (f'(V(\lambda) + u^*) - f'(u^*))\|_{L^\frac{2N}{N+2}} \ll \lambda^\frac{N-4}{2}$, which will follow from
  \begin{equation}
    \label{eq:dest-inter-3}
    \|\dot u^*\cdot f'(W_\lambda)\|_{L^\frac{2N}{N+2}} \ll \lambda^\frac{N-4}{2}.
  \end{equation}
  We check \eqref{eq:dest-inter-3} separately for $N = 3, 4, 5$. Recall that $\dot u^*$ is bounded in $H^s$.
  If $N = 3$, then $\|\dot u^*\|_{L^6}$ and $\|f'(W_\lambda)\|_{L^\frac 32}$ are bounded, hence \eqref{eq:dest-inter-3} follows from H\"older.
  If $N = 4$, then (by Sobolev) there exists $q > 4$ such that $\|\dot u^*\|_{L^q}$ is bounded.
  It can be checked that for $1 < p < 2$, $\|f'(W_\lambda)\|_{L^p} \ll 1$, hence \eqref{eq:dest-inter-3} follows.
  If $N = 5$, then there exists $q > 5$ such that $\|\dot u^*\|_{L^q}$ is bounded.
  It can be checked that for $\frac 54 < p < 2$, $\|f'(W_\lambda)\|_{L^p} \ll \sqrt\lambda$, hence \eqref{eq:dest-inter-3} follows.

  In the proof of \eqref{eq:gamma-err-3} we have only used the boundedness of $\dot u^*$ in $H^1$, hence it remains valid and gives the bound
  $$
  |\dot u^*, \Delta V(\lambda) + f(V(\lambda))\ra| \lesssim \lambda^\frac{N-2}{2} \ll \lambda^\frac{N-2}{4}.
  $$
\end{proof}

\begin{remark}
  \label{rem:loi-gamma}
  It is not excluded that Lemma~\ref{lem:loi-gamma-reg} holds under the assumption $\bs u^*_0 \in X^s \times H^s$, $s > \frac{N-2}{2}$,
  but I was unable to prove it because of possible oscillations of $\lambda(t)$.
  Note also that Lemma~\ref{lem:loi-gamma} could be proved for less regular $\bs u^*_0$ if we had some control of $\bs g(t)$
  in suitable (for example Strichartz) norms.

  Lemma~\ref{lem:loi-gamma-reg} implies that if $\bs u^*_0 \in X^2 \times H^2$, then Lemma~\ref{lem:loi-gamma} holds with $b_1(t) = 0$.
\end{remark}

For $t_0 \leq t < T_+$ we define
\begin{equation}
  \label{eq:loi-phi-def}
  \varphi(t) := C_\tx{I}c^*\lambda(t)^\frac{N-2}{2} - b_1(t) + 2\big(a^-(t)^2 + a^+(t)^2\big)
\end{equation}
($C_\tx{I}$ is a constant to be chosen shortly).
From \eqref{eq:loi-energy} we have
\begin{equation}
  \label{eq:loi-phi-2}
  \begin{aligned}
  \varphi(t) &:= C_\tx{I}c^*\lambda(t)^\frac{N-2}{2} + \big(E(\bs V(\lambda) + \bs u^*) - E(\bs W) - E(\bs u^*)\big) \\
  &+ \frac 12\la \vD^2 E(\bs V(\lambda) + \bs u^*)\bs g, \bs g\ra + 2\big(a^-(t)^2 + a^+(t)^2\big)+ b_2(t) + O(\|\bs g\|_\cE^3).
\end{aligned}
\end{equation}

We will consider the maximal function:
\begin{equation*}
  \varphi_\tx{M}(t) := \sup_{t \leq \tau < T_+}\varphi(\tau).
\end{equation*}
Note that $\varphi_\tx{M}: [t_0, T_+) \to \bR$ is decreasing, $\lim_{t \to T_+} \varphi_\tx{M}(t) = 0$
  and $0 \geq \varphi_\tx{M}'(t) \geq \min(0, \varphi'(t))$ almost everywhere.
\begin{corollary}
  \label{cor:loi-coercivity}
  Let $s > \frac{N-2}{2}$ and $s \geq 1$. For $t_0 \leq t < T_+$ there holds
  \begin{equation*}
    \varphi_\tx{M}(t) \sim \sup_{t \leq \tau < T_+}n(\bs g(\tau), \lambda(\tau))^2.
  \end{equation*}
\end{corollary}
\begin{proof}
    Lemma~\ref{lem:loi-coer} and \eqref{eq:loi-energy} yield $|\la \vD E(\bs V(\lambda) + \bs u^*), \bs g\ra| \lesssim n(\bs g, \lambda)^2$,
    hence from Lemma~\ref{lem:loi-gamma} we have
    \begin{equation}
      \label{eq:b1-bound}
    |b_1(t)| \lesssim \sup_{t \leq \tau < T_+} n(\bs g(t), \lambda(t))^2.
  \end{equation}
  Let $t \in [t_0, T_+)$ and let $t_1 \in [t, T_+)$ be such that $\varphi_\tx{M}(t) = \varphi(t_1)$
    (such $t_1$ exists by the definition of $\varphi_\tx{M}$).
    Using \eqref{eq:b1-bound} we obtain
    \begin{equation*}
      \varphi_\tx{M}(t) = \varphi(t_1) \lesssim \sup_{t_1 \leq \tau < T_+} n(\bs g(\tau), \lambda(\tau))^2 \leq \sup_{t \leq \tau < T_+} n(\bs g(\tau), \lambda(\tau))^2.
    \end{equation*}

    Now let $t_2 \in [t, T_+)$ be such that $\sup_{t \leq \tau < T_+}n(\bs g(\tau), \lambda(\tau))^2 = n(\bs g(t_2), \lambda(t_2))^2$.
      From Lemma~\ref{lem:coer} and the fact that $\|\bs V(\lambda) + \bs u^* - \bs W_\lambda\|_\cE$ is small we obtain
  \begin{equation}
    \label{eq:loi-phi-3}
    \frac{1}{2}\la \vD^2 E(\bs V(\lambda(t_2)) + \bs u^*(t_2))\bs g(t_2), \bs g(t_2)\ra + 2\big(a^-(t_2)^2 + a^+(t_2)^2\big) \gtrsim \|\bs g(t_2)\|_\cE^2.
  \end{equation}
  From Lemma~\ref{lem:loi-coer}, if we choose $C_\tx{I}$ large enough, then $C_\tx{I}c^*\lambda^\frac{N-2}{2} + E(\bs V(\lambda) + \bs u^*) - E(\bs W) - E(\bs u^*) \gtrsim c^*\lambda^\frac{N-2}{2}$,
  hence \eqref{eq:loi-phi-2} and \eqref{eq:loi-phi-3} yield
  \begin{equation*}
    \varphi(t_2) - b_2(t_2) \gtrsim n(\bs g(t_2), \lambda(t_2))^2.
  \end{equation*}
  From Lemma~\ref{lem:loi-gamma} we have $|b_2(t_2)| \leq \sqrt{c_0}\cdot \sup_{t_2\leq \tau < T_+} n(\bs g(\tau), \lambda(\tau))^2 = \sqrt{c_0}\cdot n(\bs g(t_2), \lambda(t_2))^2$,
  hence we obtain
  \begin{equation*}
    \varphi_\tx{M}(t) \geq \varphi(t_2) \gtrsim n(\bs g(t_2), \lambda(t_2))^2 = \sup_{t\leq \tau < T_+} n(\bs g(\tau), \lambda(\tau))^2,
  \end{equation*}
  provided that $c_0$ is small enough. 
\end{proof}

\subsection{Differential inequalities and conclusion}
\label{ssec:loi-proof}
\begin{lemma}
  \label{lem:loi-a}
  There exists a constant $C_a$ such that for $T_+ - t$ small enough there holds
    \begin{equation*}
      |a^+(t)| \leq C_a \cdot \sup_{t \leq \tau < T_+}n(\bs g(\tau), \lambda(\tau))^2,\qquad |a^-(t)| \leq C_a \cdot \sup_{t\leq \tau < T_+}n(\bs g(\tau), \lambda(\tau))^2.
    \end{equation*}
\end{lemma}
\begin{proof}
  It follows from \eqref{eq:loi-alphap} that there exists $C_1 > 0$ such that
  \begin{equation}
    \label{eq:loi-ap-destab}
    |a^+(t)| \geq C_1 \cdot n(\bs g(t), \lambda(t))^2 \quad\Rightarrow\quad \dd t|a^+(t)| \geq \frac{\nu}{2\lambda(t)}|a^+(t)|.
  \end{equation}
  Suppose that
  $$|a^+(t)| \geq 2C_1 \cdot \sup_{t \leq \tau < T_+}n(\bs g(\tau), \lambda(\tau))^2$$
  and suppose that $t_1 \in [t, T_+)$ is the smallest time such that
    $$|a^+(t_1)| \leq C_1 \cdot \sup_{t_1 \leq \tau < T_+}n(\bs g(\tau), \lambda(\tau))^2.$$
    Clearly $t_1 > t$. The function on the right hand side is decreasing with respect to $t_1$, hence $\dd t|a^+(t)|_{t = t_1} \leq 0$.
    This contradicts \eqref{eq:loi-ap-destab}, hence for all $t' \in [t, T_+)$ we have
      \begin{equation}
        \label{eq:loi-ap-destab-2}
        |a^+(t')| \geq C_1 \cdot n(\bs g(t'), \lambda(t'))^2.
      \end{equation}
      Observe that
      \begin{equation}
      \label{eq:temps-renorm}
      \int_{t}^{T_+}\frac{1}{\lambda(\tau)}\ud \tau \gtrsim \int_t^{T_+}\frac{|\lambda'(\tau)|}{\lambda(\tau)}\ud\tau = +\infty.
    \end{equation}
    From \eqref{eq:loi-ap-destab-2}, \eqref{eq:loi-ap-destab} and \eqref{eq:temps-renorm} we obtain
    $|a^+(t)| \to +\infty$ as $t \to T_+$, a contradiction.

      We will now consider $a^-(t)$, which is less straightforward.
        It follows from \eqref{eq:loi-alphap} that there exists $C_2 > 0$ such that
  \begin{equation}
    \label{eq:loi-am-stab}
    |a^-(t)| \geq C_2 \cdot n(\bs g(t), \lambda(t))^2 \quad\Rightarrow\quad \dd t|a^-(t)| \leq -\frac{\nu}{2\lambda(t)}|a^-(t)|.
  \end{equation}
  From Corollary~\ref{cor:loi-coercivity} we obtain existence of a constant $C_3 > 0$ such that
  \begin{equation}
    \label{eq:C3}
    |a^-(t)| \geq C_3 \cdot \varphi_\tx{M}(t)\quad\Rightarrow\quad |a^-(t)| \geq C_2 \cdot n(\bs g(t), \lambda(t))^2
  \end{equation}
  and a constant $C_4 > 0$ such that
  \begin{equation*}
    |a^-(t)| \geq C_4 \cdot \sup_{t\leq \tau < T_+}n(\bs g(t), \lambda(t))^2\quad\Rightarrow \quad |a^-(t)| \geq 2C_3 \cdot \varphi_\tx{M}(t).
  \end{equation*}
  Suppose that $t \in [t_0, T_+)$ is such that
    \begin{equation}
      \label{eq:loi-am-seq}
    |a^-(t)| \geq C_4 \cdot \sup_{t \leq \tau < T_+}n(\bs g(\tau), \lambda(\tau))^2
  \end{equation}
  and let $t_1 \in [t_0, t]$ be the smallest time such that for $t' \in [t_1, t]$ there holds
      \begin{equation}
        \label{eq:loi-am-past}
        |a^-(t')| \geq C_3 \cdot \varphi_\tx{M}(t').
      \end{equation}
  Of course $t_1 < t$. Suppose that $t_1 > t_0$. This implies
  \begin{equation*}
    -\frac{C_2\nu}{\lambda(t_1)}n(\bs g(t_1), \lambda(t_1))^2 \geq -\frac{\nu}{2\lambda(t_1)}|a^-(t_1)| \geq \dd t|a^-(t)|_{t = t_1} \geq C_3 \cdot \varphi_\tx{M}'(t_1)
  \end{equation*}
  (we use respectively \eqref{eq:C3}, \eqref{eq:loi-am-stab} and the definition of $t_1$).

  However, $|\varphi_\tx{M}'(t_1)| \leq |\varphi'(t_1)| \ll \frac{1}{\lambda(t_1)}n(\bs g(t_1), \lambda(t_1))^2$, as is easily seen from \eqref{eq:loi-phi-def}.
  The contradiction shows that $t_1 = t_0$, hence \eqref{eq:loi-am-past} holds for $t' \in [t_0, t]$.
  This means that if there exist times $t$ arbitrarily close to $T_+$ such that \eqref{eq:loi-am-seq} holds, then \eqref{eq:loi-am-past} is true for $t' \in [t_0, T_+)$.
    From \eqref{eq:loi-am-stab} and \eqref{eq:C3} we deduce that for $t \in [t_0, T_+)$ there holds
      \begin{equation*}
        |a^-(t)| \leq |a^-(t_0)|\cdot \exp\Big(-\int_{t_0}^t\frac{\nu\ud t}{2\lambda(t)}\Big).
      \end{equation*}
      By \eqref{eq:C3} and \eqref{eq:loi-mod}, this implies
      \begin{equation*}
        |\lambda'(t)| \lesssim \exp\Big(-\int_{t_0}^t\frac{\nu\ud t}{4\lambda(t)}\Big).
      \end{equation*}
      Dividing both sides by $\lambda(t)$ and integrating we get a contradiction.

  We have proved the lemma with $C_a := \max(2C_1, C_4)$.
\end{proof}
By modifying $t_0$ we can assume that Lemma~\ref{lem:loi-a} holds for $t \in [t_0, T_+)$.
\begin{proof}[Proof of Theorem~\ref{thm:loi}]
  We define
  \begin{equation*}
    \wt \varphi(t) := C_\tx{I}c^*\lambda(t)^\frac{N-2}{2} - b_1(t),\qquad \wt \varphi_\tx{M}(t) := \sup_{t \leq \tau < T_+} \wt \varphi(\tau).
  \end{equation*}
  From Lemma~\ref{lem:loi-a} and Corrolary~\ref{cor:loi-coercivity}, it is clear that
  \begin{equation}
    \label{eq:loi-equiv}
    \wt \varphi_\tx{M}(t) \sim \sup_{t \leq \tau < T_+}n(\bs g(\tau), \lambda(\tau))^2.
  \end{equation}
  We will consider first the case $N \in \{4, 5\}$.
  Using \eqref{eq:b1p} and \eqref{eq:loi-equiv} we obtain the following differential inequality for $t \in [t_0, T_+)$:
    \begin{equation}
      \label{eq:main-diff-ineq}
      |\wt \varphi_\tx{M}'(t)| \leq |\wt \varphi'(t)| \lesssim c^* \lambda(t)^\frac{N-4}{2}\|\bs g(t)\|_\cE \lesssim (c^*)^\frac {2}{N-2}\wt \varphi_\tx{M}(t)^\frac{3N-10}{2(N-2)}.
    \end{equation}
    Integrating this inequality we find
    \begin{equation*}
      \wt \varphi_\tx{M}(t) \lesssim (c^*)^\frac{4}{6-N}(T_+ - t)^\frac{2(N-2)}{6-N}.
    \end{equation*}
    To finish the proof, recall that $c^*\lambda(t)^\frac {N-2}{2} \lesssim \wt \varphi_\tx{M}(t)$ by Corollary~\ref{cor:loi-coercivity}.

    Consider now the case $N = 3$. The problem is that $N - 4 < 0$, hence we cannot write
    $(c^*\lambda(t))^\frac{N-4}{2} \lesssim \wt \varphi_\tx{M}^\frac{N-4}{2(N-2)}$, as we did in the previous proof.
    Instead, we just have
    \begin{equation*}
      |\wt \varphi_\tx{M}'(t)| \lesssim c^*\lambda(t)^{-\frac 12}\cdot \sqrt{\wt \varphi_\tx{M}(t)}.
    \end{equation*}
    Integrating between $t$ and $T_+$ we obtain
    \begin{equation*}
      \sqrt[4]{\lambda(t)} \lesssim \sqrt{c^*}\int_t^{T_+}\frac{\ud \tau}{\sqrt{\lambda(\tau)}}.
    \end{equation*}
    This is again a differential inequality. It yields \eqref{eq:moyenne}.
\end{proof}
\begin{remark}
  \label{rem:N3-cont}
  In the case $N = 3$ and $\bs u^* \in X^2 \times H^2$, we can prove \eqref{eq:loi-borne} for continuous time, not only for a sequence.
  Indeed, in this case one can take $b_1(t) = 0$ (see Remark~\ref{rem:loi-gamma}), hence $\wt \varphi_\tx{M}(t) = C_\tx{I}c^*\sqrt{\lambda(t)}$.
  If $t \in [t_0, T_+)$ is such that $\lambda(t) < \sup_{t \leq \tau < T_+}\lambda(\tau)$, then obviously $\wt \varphi_\tx{M}'(t) = 0$.
    If $\lambda(t) = \sup_{t \leq \tau < T_+} \lambda(\tau)$, then $c^*\sqrt{\lambda(t)} \sim \wt \varphi_\tx{M}(t)$,
    hence the proof of \eqref{eq:main-diff-ineq} applies. The end of the proof is the same as in the case $N \in \{4, 5\}$.
  \end{remark}

\begin{proof}[Proof of Theorem~\ref{thm:neg}]
  Let $t \in [t_0, T_+)$ be such that $n(\bs g(t), \lambda(t)) = \sup_{t \leq \tau < T_+} n(\bs g(\tau), \lambda(\tau))$.
    From \eqref{eq:inter-positive} and Lemma~\ref{lem:coer} we get
  \begin{equation*}
    \begin{aligned}
      \big(E(\bs V(\lambda) + \bs u^*)-E(\bs W) - E(\bs u^*)\big) + \frac{1}{2}\la \vD^2 E(\bs V(\lambda) + \bs u^*)\bs g, \bs g\ra
      +2((a^-)^2 + (a^+)^2) \gtrsim n(\bs g, \lambda)^2.
    \end{aligned}
  \end{equation*}
  But due to Lemma~\ref{lem:loi-a}, the last term on the right hand side can be omitted.
  This is in contradiction with \eqref{eq:loi-energy} and Lemma~\ref{lem:loi-gamma-reg}.
\end{proof}

\appendix
\section{Cauchy theory in higher regularity}
\label{sec:cauchy}
In this section we prove some facts about propagation of regularity for \eqref{eq:nlw}, which are applied to $\bs u^*(t)$ in the main text.
As in \cite[Appendix B]{moi15p}, the proofs rely on the classical \emph{energy estimates}:
\begin{proposition*}
  Let $s \geq 0$, $t_0 \in [T_1, T_2]$, $g \in L^1(I, H^s)$ and $\bs u_0 \in X^s \times H^s$.
  Then the solution of the linear wave equation $(\partial_{tt} - \Delta) u = g$ with initial data $\bs u(t_0) = \bs u_0$ satifies
  \begin{equation*}
    \|\bs u(t)\|_{X^s \times H^s} \leq \|\bs u_0\|_{X^s \times H^s} + \Big|\int_{t_0}^t\|g(\tau)\|_{H^s}\ud\tau\Big|,\qquad \forall t\in[T_1, T_2].
  \end{equation*}\qed
\end{proposition*}

\begin{proposition}
  \label{prop:cauchy-N34}
  Let $N \in \{3, 4\}$, $s > \frac{N-2}{2}$ and $\bs u_0 \in X^s \times H^s$. There exist $t_1 < t_0 < t_2$ such that the solution $\bs u(t)$ of \eqref{eq:nlw} satisfies
  \begin{equation*}
    \bs u \in C([t_1, t_2], X^s \times H^s).
  \end{equation*}
\end{proposition}
\begin{proof}
  This is a standard application of the energy estimates and the Fixed Point Theorem, using the fact that $f(u)$ is a monomial and $X^s \hookrightarrow L^\infty$.
  We skip the details.
%
\end{proof}

  In the rest of this section we consider \eqref{eq:nlw} in dimension $N = 5$. In this case the nonlinearity $f(u) = |u|^\frac 43 u$ is not smooth.
We will use the following regularization: $$f_n(u) := \big(1-\chi(nu)\big)f(u),\qquad n \in \{1, 2, 3,\ldots\},$$
where
$$
\chi \in C^\infty,\quad\chi(-u) = \chi(u),\quad \chi(u) = 1\text{ for }u\in[-1,1],\quad \supp \chi \subset [-2,2].
$$

In the proof of the next result we will use the Fractional Leibniz Rule and the Fractional Chain Rule in the form given in \cite[Propositions 3.1, 3.3]{ChWe91}:
  \begin{proposition}
    \label{prop:chain-rule}
    ~
    \begin{itemize}
      \item
        If $\Psi \in C^1$, $0 < \alpha < 1$ and $1 < p, p_1, p_2$ are such that $\frac{1}{p} = \frac{1}{p_1} + \frac{1}{p_2}$, then
        \begin{equation*}
          \||\grad|^\alpha \Psi(u)\|_{L^p} \lesssim \|\Psi'(u)\|_{L^{p_1}}\cdot \||\grad|^\alpha u\|_{L^{p_2}}.
        \end{equation*}
      \item
        If $0 < \alpha < 1$ and $1 < p, p_1, p_2, \wt p_1, \wt p_2$ are such that $\frac 1p = \frac{1}{p_1} + \frac{1}{p_2} = \frac{1}{\wt p_1} + \frac{1}{\wt p_2}$, then
        \begin{equation*}
          \||\grad|^\alpha(uv)\|_{L^p} \lesssim \||\grad|^\alpha u\|_{L^{p_1}}\cdot \|v\|_{L^{p_2}} + \|u\|_{L^{\wt p_1}}\cdot \||\grad|^\alpha v\|_{L^{\wt p_2}}.
        \end{equation*}
    \end{itemize} \qed
  \end{proposition}
\begin{remark}
  \label{rem:chain-rule}
  In \cite{ChWe91}, the Leibniz Rule and the Chain Rule are proved in the case of one space dimension,
  and necessary changes in order to carry out a proof in arbitrary dimension are indicated.
  In the present paper we use this result in dimension 5, but only for radial functions,
  and it can be verified that the Leibniz Rule and the Chain Rule for radial functions is a consequence of the one-dimensional result.
\end{remark}

\begin{lemma}
  \label{lem:cauchy-nonlin}
  Let $N = 5$ and $1 \leq s \leq 2$. The following estimates hold (with constants which may depend on $s$):
  \begin{align}
    \|f(u) - f_n(u)\|_{H^1} &\leq c_n\big(1+f(\|u\|_{X^1})\big),\qquad \text{with}\ c_n \to 0\ \text{as}\ n \to +\infty, \label{eq:f-fn_X1} \\
    \|f(u) - f(v)\|_{H^1} &\lesssim \|u-v\|_{X^1}\cdot \big(f'(\|u\|_{X^1}) + f'(\|v\|_{X^1})\big), \label{eq:f_u-v_X1} \\
    \|f_n(u) - f_n(v)\|_{H^1} &\lesssim \|u-v\|_{X^1}\cdot \big(f'(\|u\|_{X^1}) + f'(\|v\|_{X^1})\big), \label{eq:fn_u-v_X1} \\
    \|f(u)\|_{H^s} &\lesssim f(\|u\|_{X^s}), \label{eq:f_u_Xs} \\
    \|f_n(u)\|_{H^s} &\lesssim f(\|u\|_{X^s}), \label{eq:fn_u_Xs} \\
    \|f_n(u) - f_n(v)\|_{H^s} &\leq C_n\|u - v\|_{X^s}\cdot\big(1 + f'(\|u\|_{X^s}) + f'(\|v\|_{X^s})\big),\qquad C_n > 0, \label{eq:fn_u-v_Xs}
  \end{align}
  where the sign $\lesssim$ means that the constant is independent of $n$.
\end{lemma}
\begin{proof}
  A simple computation shows that
  \begin{gather}
    |f_n(u)| \leq |f(u)|,\quad |f_n'(u)| \lesssim |f'(u)|,\quad |f_n''(u)| \lesssim |f''(u)|, \label{eq:fn-deriv}\\
    f_n \to f\quad \text{in }C^2(\bR), \label{eq:fn-conv}\\
    |f_n'''(u)| \lesssim n^\frac 23\label{eq:fn-3deriv}.
  \end{gather}
  We have $$\|\grad(f(u) - f_n(u))\|_{L^2} = \|(f'(u) - f_n'(u))\grad u\|_{L^2} \leq \|f' - f_n'\|_{L^\infty}\cdot \|u\|_{H^1},$$
  which is acceptable due to \eqref{eq:fn-conv}.

  In order to bound $\|f(u) - f_n(u)\|_{L^2}$, we interpolate between $\|f-f_n\|_{L^\infty}$ and
  $$
  \|f(u) - f_n(u)\|_{L^\frac{10}{7}} \lesssim f(\|u\|_{L^\frac{10}{3}}) \lesssim f(\|u\|_{H^1}).
  $$
  This proves \eqref{eq:f-fn_X1}.

  Estimate \eqref{eq:f_u-v_X1} is a part of \cite[Lemma B.3]{moi15p} and the proof of \eqref{eq:fn_u-v_X1} is analogous.

  From the Sobolev inequality we get $\|f_n(u)\|_{L^2} \leq \|f(u)\|_{L^2} \leq f(\|u\|_{L^\frac{14}{3}}) \lesssim f(\|u\|_{X^s})$,
  hence in order to prove \eqref{eq:f_u_Xs} and \eqref{eq:fn_u_Xs} it suffices to check that
  \begin{equation}
    \label{eq:energy-deriv-1}
    \||\grad|^s(f(u))\|_{L^2} \lesssim f(\|u\|_{X^s}),\qquad  \||\grad|^s(f_n(u))\|_{L^2} \lesssim f(\|u\|_{X^s}).
  \end{equation}
  For $s \in \{1, 2\}$ this is an easy algebraic computation which we will skip. For $1 < s < 2$ we use Proposition~\ref{prop:chain-rule}:
  \begin{equation}
    \label{eq:energy-deriv-2}
    \begin{aligned}
      \||\grad|^s(f(u))\|_{L^2} &= \||\grad|^{s-1}\grad(f(u))\|_{L^2} = \||\grad|^{s-1}(f'(u)\grad u)\|_{L^2} \\
      &\lesssim\||\grad|^{s-1}\grad u\|_{L^\frac{10}{3}}\cdot \|f'(u)\|_{L^5} + \||\grad|^{s-1}(f'(u))\|_{L^5}\cdot \|\grad u\|_{L^\frac{10}{3}} \\
      &\lesssim \||\grad|^{s-1}\grad u\|_{H^1}\cdot f'(\|u\|_{L^\frac{20}{3}}) + \|f''(u)\|_{L^{10}}\cdot \||\grad|^{s-1}u\|_{L^{10}}\cdot \|\grad u\|_{H^1}  \\ &\lesssim f(\|u\|_{X^s}).
    \end{aligned}
  \end{equation}
  The second inequality in \eqref{eq:energy-deriv-2} is proved analogously.

  In order to prove \eqref{eq:fn_u-v_Xs} it suffices to check that
  \begin{equation}
    \label{eq:energy-deriv-3}
    \||\grad|^s(f_n(u) - f_n(v))\|_{L^2} \leq C_n\|u - v\|_{X^s}\cdot\big(1 + f'(\|u\|_{X^s}) + f'(\|v\|_{X^s})\big)
  \end{equation}
  (the estimate of $\|f_n(u) - f_n(v)\|_{L^2}$ is a part of \eqref{eq:fn_u-v_X1}).
  We write $$f_n(u) - f_n(v) = -(v-u)\int_0^1f_n'((1-t)u + tv)\ud t,$$ hence by the triangle inequality
  \begin{equation*}
    \||\grad|^s(f_n(u) - f_n(v))\|_{L^2} \leq \int_0^1 \big\||\grad|^s\big((u-v)f_n'((1-t)u + tv)\big)\big\|_{L^2}\ud t.
  \end{equation*}
  We will estimate the integrand for fixed $t \in [0, 1]$. We have
  \begin{equation*}
    \begin{aligned}
    \||\grad|^s\big((u-v)f_n'((1-t)u + tv)\big)\big\|_{L^2} &= \||\grad|^{s-1}\grad\big((u-v)f_n'((1-t)u + tv)\big)\big\|_{L^2} \\
    &= \big\||\grad|^{s-1}\big(\grad(u-v)\cdot f_n'((1-t)u + tv)\big)\big\|_{L^2}  \\&+ \big\||\grad|^{s-1}\big((u-v)\cdot ((1-t)\grad u + t\grad v)\cdot f_n''((1-t)u + tv)\big)\big\|_{L^2}.
  \end{aligned}
  \end{equation*}
  The first term is estimated exactly as in \eqref{eq:energy-deriv-2}, so we will only consider the second one.
  From the Leibniz Rule we obtain
  \begin{equation*}
    \begin{aligned}
    &\big\||\grad|^{s-1}\big((u-v)\cdot ((1-t)\grad u + t\grad v)\cdot f_n''((1-t)u + tv)\big)\big\|_{L^2} \\
    \lesssim &\||\grad|^{s-1}(u-v)\|_{L^{10}}\cdot \|(1-t)u + tv\|_{L^\frac{10}{3}}\cdot \|f_n''((1-t)u + tv)\|_{L^{10}} \\
    + &\|u - v\|_{L^{10}}\cdot \||\grad|^{s-1}((1-t)\grad u + t\grad v)\|_{L^\frac{10}{3}}\cdot \|f_n''((1-t)u + tv)\|_{L^{10}} \\
    + &\|u - v\|_{L^{p_1}}\cdot \|(1-t)\grad u + t\grad v\|_{L^{p_2}}\cdot \||\grad|^{s-1}f_n''((1-t)u + tv)\|_{L^{p_3}},
  \end{aligned}
  \end{equation*}
  where the exponents $p_1$, $p_2, p_3 \in (1, +\infty)$ are chosen such that $p_1 > 10$, $p_2 > \frac{10}{3}$, $p_3 < 10$, $X^s \subset L^{p_1} \cap W^{1, p_2}$ and
  $\frac 12 = \frac{1}{p_1} + \frac{1}{p_2} + \frac{1}{p_3}$.
  Estimating the first two lines is straightforward and for the last line we use the Chain Rule together with \eqref{eq:fn-3deriv}.
\end{proof}

\begin{proposition}
  \label{prop:cauchy-N5}
  Let $N = 5$, $1 \leq s \leq 2$ and $\bs u_0 \in X^s \times H^s$. There exist $t_1 < t_0 < t_2$ such that the solution $\bs u(t)$ of \eqref{eq:nlw} satisfies
  \begin{equation*}
    \bs u \in C([t_1, t_2], X^s \times H^s).
  \end{equation*}
\end{proposition}
\begin{proof}
  Using \eqref{eq:f_u_Xs} for $s = 1$ and \eqref{eq:f_u-v_X1} one obtains by a standard procedure that there exists a unique maximal solution
  $$
  \bs u \in C([T_1, T_2], X^1 \times H^1),\qquad T_1 < t_0 < T_2
  $$
  and
  $$
  T_{1} > -\infty\ \Rightarrow\ \lim_{t\to T_{1}}\|\bs u_n\|_{X^1 \times H^1} = +\infty, \qquad T_{2} <+\infty\ \Rightarrow\ \lim_{t\to T_{2}}\|\bs u_n\|_{X^1 \times H^1} = +\infty,
  $$
  see \cite[Proposition B.2]{moi15p} for details.

  Consider the regularized problem for $n \in \{1, 2, 3, \ldots\}$:
\begin{equation}
  \label{eq:nlw-reg}
\bigg\{
  \begin{aligned}
    (\partial_{tt} -\Delta)u_n  &= f_n(u_n), \\
    (u_n(t_0), \partial_t u_n(t_0)) &= \bs u_0.
\end{aligned}
\end{equation}
  Using \eqref{eq:fn_u_Xs} and \eqref{eq:fn_u-v_Xs} one can show that there exists a unique maximal solution
  $$
  \bs u_n \in C([T_{1, n}, T_{2, n}], X^s \times H^s),\qquad T_{1, n} < t_0 < T_{2, n}
  $$
  and
  \begin{equation}
    \label{eq:alternative-Hs}
  T_{1, n} > -\infty\ \Rightarrow\ \lim_{t\to T_{1, n}}\|\bs u_n\|_{X^s \times H^s} = +\infty, \qquad T_{2, n} <+\infty\ \Rightarrow\ \lim_{t\to T_{2, n}}\|\bs u_n\|_{X^s \times H^s} = +\infty.
\end{equation}
From \eqref{eq:fn_u_Xs} and the energy estimate we have
\begin{equation*}
  \|\bs u_n(t)\|_{X^s \times H^s} \lesssim \|\bs u_0\|_{X^s \times H^s} + \Big|\int_{t_0}^t f(\|\bs u(\tau)\|_{X^s \times H^s})\ud \tau\Big|,
\end{equation*}
with a constant independent of $n$. This implies that there exist $\wt T_1 < t_0$, $\wt T_2 > t_0$ and a constant $C_1$ independent of $n$ such that
\begin{equation}
  \label{eq:un-bound}
  \|\bs u_n(t)\|_{X^s\times H^s} \leq C_1\qquad \forall n,\ \forall t\in[\wt T_1, \wt T_2]
\end{equation}
(in particular $\wt T_1 \geq \sup_n T_{1, n}$ and $\wt T_2 \leq \inf_n T_{2, n}$).

Now we need to verify that
\begin{equation}
  \label{eq:un-conv}
  \lim_{n\to +\infty}\|\bs u_n(t) - \bs u(t)\|_{X^1 \times H^1} = 0\qquad \forall t \in [\wt T_1, \wt T_2].
\end{equation}
To this end, we notice that $\bs u_n - \bs u$ solves the Cauchy problem:
\begin{equation}
  \label{eq:nlw-un-u}
\bigg\{
  \begin{aligned}
    (\partial_{tt} -\Delta)(u_n-u)  &= f_n(u_n)-f(u), \\
    (u_n(t_0), \partial_t u_n(t_0)) &= 0.
\end{aligned}
\end{equation}
Since $\|\bs u(t)\|_{X^1\times H^1}$ is bounded and $\|\bs u_n(t)\|_{X^1\times H^1}$ are uniformly bounded for $t \in [\wt T_1, \wt T_2]$, \eqref{eq:f-fn_X1} and \eqref{eq:fn_u-v_X1}
imply that for $t\in [\wt T_1, \wt T_2]$ there holds
$$
\|f_n(u_n(t)) - f(u(t))\|_{H^1} \leq \|f_n(u_n(t)) - f_n(u(t))\|_{H^1} + \|f_n(u(t)) - f(u(t))\|_{H^1} \lesssim \|u_n(t) - u(t)\|_{X^1} + c_n,
$$
which yields \eqref{eq:un-conv} by the energy estimate and the Gronwall inequality.

From \eqref{eq:un-bound} and \eqref{eq:un-conv} we deduce
\begin{equation*}
  \|\bs u(t)\|_{X^s \times H^s} \leq C_1,\qquad \forall t\in [\wt T_1, \wt T_2].
\end{equation*}
The function $\bs u: [\wt T_1, \wt T_2] \to X^s \times H^s$ is weakly measurable (since it is measurable as a function to $X^1 \times H^1$), hence it is measurable
and $\bs u\in L^\infty([\wt T_1, \wt T_2], X^s \times H^s)$. Using once again the energy estimate together with \eqref{eq:f_u_Xs} it is easy to see that in fact
$\bs u \in C([\wt T_1, \wt T_2], X^s \times H^s)$.

\end{proof}
\bibliographystyle{plain}
\bibliography{une-bulle}

\begin{thebibliography}{10}

\bibitem{BoWa97}
J.~Bourgain and W.~Wang.
\newblock Construction of blowup solutions for the nonlinear {S}chr{\"o}dinger
  equation with critical nonlinearity.
\newblock {\em Ann. Scuola Norm. Sup. Pisa Cl. Sci. (4)}, 25:197--215, 1997.

\bibitem{chow-hale}
S.-N. Chow and J.~K. Hale.
\newblock {\em {Methods of Bifurcation Theory}}, volume 251 of {\em Grundlehren
  der mathematischen Wissenschaften}.
\newblock Springer, 1982.

\bibitem{ChWe91}
F.~M. Christ and M.~I. Weinstein.
\newblock Dispersion of small amplitude solutions of the generalized
  {K}orteweg-de {V}ries equation.
\newblock {\em J. Funct. Anal.}, 100:87--109, 1991.

\bibitem{CKLS14p}
R.~C\^ote, C.~Kenig, A.~Lawrie, and W.~Schlag.
\newblock Profiles for the radial focusing 4d energy-critical wave equation.
\newblock {\em Preprint}, 2014.

\bibitem{DoSc14}
R.~Donninger and B.~Sch\"orkhuber.
\newblock Stable blow up dynamics for energy supercritical wave equations.
\newblock {\em Trans. Amer. Math. Soc.}, 366(4):2167--2189, 2014.

\bibitem{DoSc15p}
R.~Donninger and B.~Sch\"orkhuber.
\newblock Stable blowup for wave equations in odd space dimensions.
\newblock {\em Preprint}, 2015.

\bibitem{DKM1}
T.~Duyckaerts, C.~Kenig, and F.~Merle.
\newblock Universality of blow-up profile for small radial type {II} blow-up
  solutions of the energy-critical wave equation.
\newblock {\em J. Eur. Math. Soc.}, 13(3):533--599, 2011.

\bibitem{DKM2}
T.~Duyckaerts, C.~Kenig, and F.~Merle.
\newblock Universality of the blow-up profile for small type {II} blow-up
  solutions of the energy-critical wave equation: the nonradial case.
\newblock {\em J. Eur. Math. Soc.}, 14(5):1389--1454, 2012.

\bibitem{DKM4}
T.~Duyckaerts, C.~Kenig, and F.~Merle.
\newblock Classification of the radial solutions of the focusing,
  energy-critical wave equation.
\newblock {\em Camb. J. Math.}, 1(1):75--144, 2013.

\bibitem{DM08}
T.~Duyckaerts and F.~Merle.
\newblock Dynamics of threshold solutions for energy-critical wave equation.
\newblock {\em Int. Math. Res. Pap. IMRP}, 2008.

\bibitem{GSV92}
J.~Ginibre, A.~Soffer, and G.~Velo.
\newblock The global cauchy problem for the critical nonlinear wave equation.
\newblock {\em J. Funct. Anal.}, 110:96--130, 1992.

\bibitem{HiRa12}
M.~Hillairet and P.~Rapha{\"e}l.
\newblock Smooth type {II} blow up solutions to the four dimensional energy
  critical wave equation.
\newblock {\em Analysis {\&} PDE}, 5(4):777--829, 2012.

\bibitem{moi15p}
J.~Jendrej.
\newblock Construction of type {II} blow-up solutions for the energy-critical
  wave equation in dimension 5.
\newblock {\em Preprint}, 2015.

\bibitem{KeMe08}
C.~E. Kenig and F.~Merle.
\newblock Global well-posedness, scattering and blow-up for the energy-critical
  focusing non-linear wave equation.
\newblock {\em Acta Math.}, 201(2):147--212, 2008.

\bibitem{KrScTa09}
J.~Krieger, W.~Schlag, and D.~Tataru.
\newblock Slow blow-up solutions for the {$H\sp 1(\mathbb{R}\sp 3)$} critical
  focusing semilinear wave equation.
\newblock {\em Duke Math. J.}, 147(1):1--53, 2009.

\bibitem{MaMe01}
Y.~Martel and F.~Merle.
\newblock Instability of solitons for the critical generalized {K}orteweg-de
  {V}ries equation.
\newblock {\em Geom. Funct. Anal.}, 11:74--123, 2001.

\bibitem{MMR12-1p}
Y.~Martel, F.~Merle, and P.~Rapha{\"e}l.
\newblock Blow up for the critical {gKdV} equation {I}: dynamics near the
  soliton.
\newblock {\em Preprint}, 2012.

\bibitem{MMR12-3p}
Y.~Martel, F.~Merle, and P.~Rapha{\"e}l.
\newblock Blow up for the critical {gKdV} equation {III}: exotic regimes.
\newblock {\em Preprint}, 2012.

\bibitem{MeRa05}
F.~Merle and P.~Rapha{\"e}l.
\newblock Profiles and quantization of the blow up mass for critical nonlinear
  {S}chr{\"o}dinger equation.
\newblock {\em Commun. Math. Phys.}, 253(3):675--704, 2005.

\bibitem{MeZa03}
F.~Merle and H.~Zaag.
\newblock Determination of the blow-up rate for the semilinear wave equation.
\newblock {\em Amer. J. Math.}, 125(5):1147--1164, 2003.

\bibitem{MeZa05}
F.~Merle and H.~Zaag.
\newblock Determination of the blow-up rate for a critical semilinear wave
  equation.
\newblock {\em Math. Ann.}, 331(2):395--416, 2005.

\bibitem{ShSt94}
J.~Shatah and M.~Struwe.
\newblock Well-posedness in the energy space for semilinear wave equations with
  critical growth.
\newblock {\em Internat. Math. Res. Notices}, 7:303--309, 1994.

\end{thebibliography}

\end{document}